\documentclass{amsart}
\usepackage{mathptmx}
\usepackage{amscd,amssymb,enumerate,bm}

\newtheorem{theorem}{Theorem}[section]
\newtheorem{lemma}[theorem]{Lemma}
\newtheorem{corollary}[theorem]{Corollary}
\newtheorem{proposition}[theorem]{Proposition}

\theoremstyle{definition}

\newtheorem{example}[theorem]{Example}
\newtheorem{remark}[theorem]{Remark}

\newtheorem{question}[theorem]{Question}
\numberwithin{equation}{theorem}

\def\ara{\operatorname{ara}}
\def\height{\operatorname{height}}
\def\image{\operatorname{image}}
\def\cd{\operatorname{cd}}
\def\rad{\operatorname{rad}}
\def\rank{\operatorname{rank}}

\def\Cdot{C^\bullet}
\def\dR{\operatorname{dR}}

\def\End{\operatorname{End}}
\def\Ext{\operatorname{Ext}}
\def\Hom{\operatorname{Hom}}
\def\Hsing{H_{\mathrm{sing}}}

\def\Spec{\operatorname{Spec}}
\def\Supp{\operatorname{Supp}}
\def\Alt{\operatorname{Alt}}
\def\Sym{\operatorname{Sym}}
\def\Var{\operatorname{Var}}

\def\bsg{{\boldsymbol{g}}}
\def\bdel{{\boldsymbol{\partial}}}

\def\fraka{\mathfrak{a}}
\def\frakb{\mathfrak{b}}
\def\frakm{\mathfrak{m}}
\def\frakp{\mathfrak{p}}

\def\CC{\mathbb{C}}
\def\FF{\mathbb{F}}
\def\NN{\mathbb{N}}
\def\PP{\mathbb{P}}
\def\QQ{\mathbb{Q}}
\def\RR{\mathbb{R}}
\def\SS{\mathbb{S}}
\def\ZZ{\mathbb{Z}}

\def\calD{\mathcal{D}}
\def\calF{\mathcal{F}}

\def\calM{\mathcal{M}}
\def\calN{\mathcal{N}}
\def\calO{\mathcal{O}}

\def\grD{{^*\!D}}
\def\grtau{{^*\!\tau}}
\def\calH{{^*\!\!{\mathcal{H}}}}
\def\nil{\mathrm{nil}}
\def\red{\mathrm{red}}
\def\st{\mathrm{st}}

\def\ge{\geqslant}
\def\le{\leqslant}
\def\phi{\varphi}
\def\bar{\overline}
\def\tilde{\widetilde}
\def\del{\partial}
\def\dlim{\varinjlim}
\def\to{\longrightarrow}
\def\mapsto{\longmapsto}
\renewcommand{\mod}{\,\operatorname{mod}\,}

\begin{document}
\title[Local cohomology modules supported at determinantal ideals]{Local cohomology modules supported at\\
determinantal ideals}

\author{Gennady Lyubeznik}
\address{Department of Mathematics, University of Minnesota, 127 Vincent Hall, 206 Church~St.,
\newline Minneapolis, MN~55455, USA}
\email{gennady@math.umn.edu}

\author{Anurag K. Singh}
\address{Department of Mathematics, University of Utah, 155 South 1400 East, Salt Lake City,
\newline UT~84112, USA}
\email{singh@math.utah.edu}

\author{Uli Walther}
\address{Department of Mathematics, Purdue University, 150 N University St., West Lafayette,
\newline IN~47907, USA}
\email{walther@math.purdue.edu}

\thanks{G.L.~was supported by NSF grant DMS~1161783, A.K.S.~by NSF grants DMS~1162585 and DMS~1500613, and U.W.~by NSF grants DMS~0901123 and DMS~1401392. G.L. and A.K.S. thank the American Institute of Mathematics for supporting their collaboration. All authors were also supported by NSF grant~0932078000 while in residence at MSRI}

\maketitle

\section{Introduction}

In \cite[Corollary~6.5]{HKM}, Huneke, Katz, and Marley proved the following striking result: If~$A$ is a commutative Noetherian ring containing the field of rational numbers, with $\dim\,A\le 5$, and~$\fraka$ is the ideal generated by the size $2$ minors of an arbitrary $2\times 3$ matrix with entries from~$A$, then the local cohomology module $H^3_\fraka(A)$ equals zero. What makes this striking is that it does not follow from classical vanishing theorems as in~\cite{HL}. It is natural to ask whether the same holds for rings that do not necessarily contain the rationals, and whether such results extend to matrices and minors of other sizes. Indeed, we prove:

\begin{theorem}
\label{theorem:intro:vanish}
Let $\fraka$ be the ideal generated by the size $t$ minors of an $m\times n$ matrix with entries from a commutative Noetherian ring $A$, where $1\le t\le\min\{m,n\}$, and~$t$ differs from at least one of $m$ and $n$. If $\dim\,A<mn$, then $H^{mn-t^2+1}_\fraka(A)=0$.
\end{theorem}

The index $mn-t^2+1$ is the cohomological dimension in the case of a matrix of indeterminates $X=(x_{ij})$ over $\QQ$ by Bruns and Schw\"anzl~\cite{Bruns-Schwanzl}; specifically,
\[
H^{mn-t^2+1}_{I_t(X)}(\QQ[X])\ \neq\ 0\,,
\]
where $I_t(X)$ is the ideal generated by the size $t$ minors of the matrix $X$. Theorem~\ref{theorem:intro:vanish} implies that the asserted vanishing holds whenever the entries of the matrix are \emph{not} algebraically independent. In the case $m=2$, $n=3$, and~$t=2$, the theorem says precisely that $H^3_\fraka(A)=0$ if $\dim\,A\le 5$, as proved in~\cite{HKM} when $A$ contains the field of rational numbers. The result is straightforward when $A$ contains a field of prime characteristic, and one of the main points of the present paper is that it includes the case of rings that do not necessarily contain a field. This requires calculations of local cohomology in polynomial rings~$\ZZ[X]$; these calculations are of independent interest, and a key ingredient is proving that there is no integer torsion in the critical local cohomology modules. More generally, we prove:

\begin{theorem}
\label{theorem:intro:rational}
Let $R=\ZZ[X]$ be a polynomial ring, where $X$ is an $m\times n$ matrix of indeterminates. Let $I_t$ be the ideal generated by the size $t$ minors of $X$. Then:
\begin{enumerate}[\quad\rm(1)]
\item $H^k_{I_t}(R)$ is a torsion-free $\ZZ$-module for all integers $t,k$.
\item If $k$ differs from the height of $I_t$, then $H^k_{I_t}(R)$ is a $\QQ$-vector space.
\item Consider the $\NN$-grading on $R$ with ${[R]}_0=\ZZ$ and~$\deg x_{ij}=1$ for all $i,j$. Set $\frakm=(x_{11},\dots,x_{mn})$. If $2\le t\le\min\{m,n\}$, and $t$ differs from at least one of $m$ and $n$, then there exists a degree-preserving isomorphism
\[
H^{mn-t^2+1}_{I_t}(\ZZ[X])\ \cong\ H^{mn}_\frakm(\QQ[X])\,.
\]
\end{enumerate}
\end{theorem}

Theorem~\ref{theorem:intro:rational} is extremely useful: once we know that $H^k_{I_t}(\ZZ[X])$ is a $\QQ$-vector space, it can then be computed using the $\calD$-module algorithms of Walther~\cite{Walther:algo} or Oaku and Takayama~\cite{Oaku-Takayama}; it can also be studied using singular cohomology and comparison theorems as in \cite{Bruns-Schwanzl}, or using representation theory as in \cite{Witt, RaicuWeymanWitt, RaicuWeyman:ANT}. For example, Theorem~\ref{theorem:intro:rational} implies that the module $H^k_{I_t}(\ZZ[X])$ is nonzero precisely if $H^k_{I_t}(\CC[X])$ is nonzero; for recent results on the nonvanishing and structure of $H^k_{I_t}(\CC[X])$ in terms of Schur functors, we refer the reader to \cite{RaicuWeymanWitt, RaicuWeyman:ANT}.

As an illustration of Theorem~\ref{theorem:intro:rational}, consider a $2\times 3$ matrix of indeterminates $X$ over $\ZZ$. Then the theorem gives
\[
H^3_{I_2}(\ZZ[X])\ \cong\ H^6_\frakm(\QQ[X])\,.
\]
The first proof that $H^3_{I_2}(\ZZ[X])$ is a $\QQ$-vector space used equational identities from~\cite{Singh:MRL,Singh:Guanajuato} that were constructed using the hypergeometric series algorithms of Petkov\v sek, Wilf, and Zeilberger~\cite{PWZ}; the module $H^3_{I_2}(\ZZ[X])$ is computed as well in Kashiwara and Lauritzen,~\cite{Kashiwara-Lauritzen}. The approach in the present paper is as follows: Let~$p$ be a prime integer; we study the annihilator of $p$ in $H^3_{I_2}(\ZZ[X])$ as a $\calD$-module, and use a duality result for $\calD$-modules, Theorem~\ref{theorem:quotient:graded}, to show that it vanishes. This requires Lyubeznik's theory of $\calF$-modules \cite{Lyubeznik:Crelle}, and also differential operators over~$\ZZ[X]$, $\FF_p[X]$ and $\QQ[X]$. These techniques work in good generality.

Section~\ref{sec:F:modules} develops the theory of graded $\calF$-modules and $\calD$-modules; the key result for our applications is Theorem~\ref{theorem:quotient:graded}, but in the process, we arrive at several results of independent interest: e.g., for a polynomial ring $R$ over a separably closed field of prime characteristic, we prove that the $\calF$-module $H^{\dim R}_\frakm(R)$ is an injective object in the category of \emph{graded} $\calF$-finite modules, Corollary~\ref{corollary:injective}. By an example of Ma, the module $H^{\dim R}_\frakm(R)$ need not be an injective object in the category of $\calF$-finite modules, see~\cite[Example~4.8]{Ma}.

Some preliminary results on local cohomology are recorded in Section~\ref{sec:lc}; this includes an interpretation of Bass numbers of $\frakm$-torsion local cohomology modules as ranks of singular cohomology groups, Theorem~\ref{theorem:comparison}. Our study of the local cohomology of polynomial rings over $\ZZ$ has its origins in a question of Huneke \cite{Huneke:Sundance} on the associated primes of local cohomology modules; this, as well, is discussed in Section~\ref{sec:lc}.

The proof of Theorem~\ref{theorem:intro:rational} occupies Section~\ref{sec:determinantal}, and in Section~\ref{sec:vanishing} we prove a vanishing theorem that subsumes Theorem~\ref{theorem:intro:vanish}. In addition to determinantal ideals, our methods extend to ideals generated by Pfaffians of alternating matrices, Section~\ref{sec:alt}, and minors of symmetric matrices, Section~\ref{sec:sym}. For these, we use Barile's computations of arithmetic rank from \cite{Barile:ara}. Section~\ref{sec:ara} deals with questions on arithmetic rank related to the vanishing theorems proved in our paper.

To assist the reader, we mention that $R$ will typically denote a commutative Noetherian ring that is regular, and $A$ an arbitrary commutative Noetherian ring.

\section{Graded $\calF$-modules}
\label{sec:F:modules}

Let $R=\FF[x_1,\dots,x_n]$ be the polynomial ring in variables $x_1,\dots,x_n$ over a field $\FF$ of characteristic $p>0$. We fix the standard $\NN$-grading on $R$ where ${[R]}_0=\FF$ and $\deg x_i=1$ for each $i$. By a \emph{graded module $M$}, we mean a $\ZZ$-graded module; we use ${[M]}_k$ for the graded component of $M$ in degree $k$, and $M(j)$ to denote the module $M$ with the shifted grading
\[
{[M(j)]}_k\ =\ {[M]}_{j+k}\,.
\]

\subsection*{$\calF$-modules}

The concept of $\calF$-modules was introduced in \cite{Lyubeznik:Crelle}. Set $R'$ to be the~$R$-bimodule that agrees with $R$ as a left $R$-module, and has the right $R$-action
\[
r'r=r^pr'\qquad\text{ for }\ r\in R \ \text{ and }\ r'\in R'\,.
\]
For an $R$-module $M$, we set $F(M)=R'\otimes_R M$; this is an $R$-module via the left $R$-module structure on $R'$.

An \emph{$\calF$-module} is a pair $(\calM,\theta)$, where $\calM$ is an $R$-module, and $\theta\colon\calM\to F(\calM)$ is an~$R$-module isomorphism called the \emph{structure isomorphism}; we sometimes suppress~$\theta$ from the notation. A \emph{morphism of $\calF$-modules} $(\calM,\theta)\to(\calM',\theta')$ is an $R$-module homomorphism $\phi\colon\calM\to\calM'$ that commutes with the structure isomorphisms, i.e.,
\[
\theta'\circ\phi\ =\ F(\phi)\circ\theta\,,
\]
see \cite[Definition~1.1]{Lyubeznik:Crelle}. With these definitions, $\calF$-modules form an Abelian category.

Graded $\calF$-modules have been studied previously in \cite[Chapter~4.3.3]{Blickle} and \cite{YiZhang, Ma-Zhang}. In this section, we establish properties of graded $\calF$-modules that will be used later in the paper; we believe these are also of independent interest. 

If $M$ is a graded $R$-module, then there is a natural grading on $F(M)=R'\otimes_RM$ given by
\[
\deg(r'\otimes m)\ =\ \deg r'+p\cdot\deg m\,,
\]
for homogeneous elements $r'\in R'$ and $m\in M$. With this grading, a \emph{graded $\calF$-module} is an $\calF$-module $(\calM,\theta)$ where $\calM$ is a graded $R$-module, and~$\theta$ is \emph{degree-preserving}, i.e., $\theta$ maps homogeneous elements to homogeneous elements of the same degree. A \emph{morphism of graded $\calF$-modules} is a degree-preserving morphism of $\calF$-modules. It is not hard to see that graded $\calF$-modules form an Abelian subcategory of the category of $\calF$-modules.

The ring $R$ has a natural graded $\calF$-module structure with structure morphism
\[
R\to R'\otimes_RR\,,\qquad r\mapsto r\otimes1\,.
\]
Let $\frakm$ be the homogeneous maximal ideal of $R$. Let $f$ denote the Frobenius action on the local cohomology module $H^n_\frakm(R)$; the image of $f$ generates $H^n_\frakm(R)$ as an $R$-module. Thus, the following structure morphism defines a graded $\calF$-module structure on $H^n_\frakm(R)$:
\[
H^n_\frakm(R)\to R'\otimes_RH^n_\frakm(R)\,,\qquad rf(\eta)\mapsto r\otimes\eta\,.
\]

\subsection*{$\calD$-modules}

The ring $\calD=\calD_{\,\FF}(R)$ of $\FF$-linear differential operators on $R$ is the subring of the ring $\End_\FF R$ generated by $R$ and all operators of the form
\[
\del_i^{[t]}\ =\ \frac{1}{t!}\frac{\del^t}{\del x_i^t}\,,
\]
see~\cite[Th\'eor\`eme~16.11.2]{EGA4}. In fact, $\calD$ is a free $R$-module, with basis
\[
\del_1^{[t_1]}\ \cdots\ \del_n^{[t_n]}\qquad\text{ for } \ (t_1,\dots,t_n)\in\NN^n\,.
\]

As shown in \cite[page~115--116]{Lyubeznik:Crelle}, each $\calF$-module carries a natural $\calD$-module structure; there exists a functor
\[
\xi\colon\calF\text{-mod}\to\calD\text{-mod}
\]
from the category of $\calF$-modules to the category of $\calD$-modules, where the $\calF$-module~$\calM$ and the $\calD$-module $\xi(\calM)$ have the same underlying $R$-module structure, and the maps $\phi\colon\calM\to\calM'$ and~$\xi(\phi)\colon\xi(\calM)\to\xi(\calM')$ agree as maps of sets.

Following \cite{Ma-Zhang}, for each positive integer $k$, we set $E_k$ to be the differential operator
\[
\sum_{\substack{t_i\ge0\\ t_1+\cdots+t_n=k}}x_1^{t_1}\cdots x_n^{t_n}\del_1^{[t_1]}\cdots\del_n^{[t_n]}\,,
\]
which is the \emph{$k$-th Euler operator}; note that
\[
E_1\ =\ x_1\del_1+\cdots+x_n\del_n
\]
is the classical Euler operator. By \cite[Theorem~4.4]{Ma-Zhang}, if $\calM$ is a graded $\calF$-module, then the $\calD$-module $\xi(\calM)$ is \emph{Eulerian}, which, by definition, means that
\[
E_k(m)\ =\ \binom{\deg m}{k}\, m
\]
for each positive integer $k$ and each homogeneous element $m$ of $\calM$.

We record an elementary lemma:

\begin{lemma}
\label{lemma:binom}
Let $d$ be a positive integer, with base $p$ expansion
\[
d\ =\ s_0+s_1p+\dots+s_tp^t\,,\qquad\text{ where }\ 0\le s_e\le p-1\ \text{ for each }\ e\,.
\]
Then, for each $e$, the binomial coefficient $\binom{d}{p^e}$ is congruent to $s_e$ modulo $p$.
\end{lemma}

\begin{proof}
Working in the polynomial ring $\FF_p[z]$, the binomial coefficient $\binom{d}{p^e}\mod p$ is the coefficient of $z^{p^e}$ in the expansion of $(1+z)^d$. Note that
\[
(1+z)^d\ =\ (1+z)^{\sum_e s_ep^e}\ =\ \prod_e(1+z)^{s_ep^e}\ =\ \prod_e{\left(1+z^{p^e}\right)}^{s_e}\ =\ \prod_e\sum_i\binom{s_e}{i}z^{ip^e}\,.
\]
When expanding the right hand side, each $z^\ell$ appears at most once by the uniqueness of the base $p$ expansion of $\ell$; specifically, $z^{p^e}$ occurs with coefficient $\binom{s_e}{1}=s_e$.
\end{proof}

\begin{proposition}
\label{proposition:subcategory}
The category of graded $\calF$-modules is a full subcategory of the category of $\calF$-modules, i.e., every $\calF$-module morphism of graded $\calF$-modules is degree-preserving.

Let $\calN\subset\calM$ be $\calF$-modules. If $\calM$ is a graded $\calF$-module, then $\calN$ and $\calM/\calN$ are graded $\calF$-modules.
\end{proposition}

By the above proposition, the category of graded $\calF$-modules is closed, in the category of $\calF$-modules, under the formation of subquotients; it is not closed under extensions; see Example~\ref{example:Ma}, which uses~\cite[Example~4.8]{Ma}.

\begin{proof}
Let $\phi\colon\calM\to\calM'$ be an $\calF$-module map, where $\calM,\calM'$ are graded $\calF$-modules; we need to show that $\phi$ is degree-preserving. Let $m$ be a homogeneous element of~$\calM$ of degree $d$. Express $\phi(m)$ as a sum of homogeneous elements,
\[
\phi(m)\ =\ m_1+\dots+m_v\,,
\]
where $m_i\in\calM'$ is homogeneous of degree $d_i$, and the integers $d_i$ are pairwise distinct. Since $\xi$ is a functor and $\xi(\calM)$ and $\xi(\calM')$ agree with $\calM$ and $\calM'$ respectively as sets, the map $\phi$ is a $\calD$-module map. It follows that
\[
\phi(E_k(m))\ =\ E_k(\phi(m))\qquad\text{ for each }\ k\ge 1\,.
\]
Expanding each side, one has
\[
\sum_i\binom{d}{k}m_i\ =\ \sum_i\binom{d_i}{k}m_i\,,
\]
and hence
\[
\binom{d}{k}\equiv\binom{d_i}{k}\ \mod p\qquad\text{ for each }\ i,k\,.
\]
Lemma~\ref{lemma:binom} implies that $d_i=d$ for each $i$, and hence also that $v=1$. Thus, the element~$\phi(m)$ is homogeneous of degree $d$, which proves the first assertion.

We next show that $\calN$ is a graded $\calF$-module. Given $m\in\calN$, write it as a sum of homogeneous elements
\[
m\ =\ m_1+\dots+m_v\,,
\]
where $m_i\in\calM$ is homogeneous of degree $d_i$, and the integers $d_i$ are pairwise distinct; we need to show that $m_i\in\calN$ for each $i$. By a slight abuse of notation we denote $\xi(\calN)$ and~$\xi(\calM)$ by $\calN$ and $\calM$ respectively. Since $\calN$ is a $\calD$-submodule of $\calM$, and $m\in\calN$, it follows that $E_k(m)\in\calN$ for each $k\ge1$. But then
\[
\binom{d_v}{k}m-E_k(m)\ =\ \sum_{i=1}^v\left[\binom{d_v}{k}-\binom{d_i}{k}\right]m_i \ =\ \sum_{i=1}^{v-1}\left[\binom{d_v}{k}-\binom{d_i}{k}\right]m_i
\]
is an element of $\calN$ for each $k\ge1$; by Lemma~\ref{lemma:binom}, $\binom{d_v}{k}-\binom{d_i}{k}$ is nonzero for some choice of~$k$. As the displayed element is a sum of at most $v-1$ homogeneous elements, an induction on $v$ shows that $m_i\in\calN$ for each $i$. The final assertion, namely that $\calM/\calN$ is a graded $\calF$-module, follows immediately.
\end{proof}

The proof of the previous proposition also yields:

\begin{proposition}
A $\calD$-module map between Eulerian $\calD$-modules is degree-preserving.

Let $\calN\subset\calM$ be $\calD$-modules. If $\calM$ is Eulerian, then so are $\calN$ and $\calM/\calN$.
\end{proposition}

\subsection*{$\calF$-finite modules}

An $\calF$-module $(\calM,\theta)$ is \emph{$\calF$-finite} if $\calM$ is the direct limit of the top row in the commutative diagram
\[
\begin{CD}
M@>\beta>>F(M)@>F(\beta)>>F^2(M)@>>>\cdots\\
@V\beta VV @VF(\beta)VV @VF^2(\beta)VV\\
F(M)@>F(\beta)>>F^2(M)@>F^2(\beta)>>F^3(M)@>>>\cdots
\end{CD}
\]
where $M$ is a finitely generated $R$-module, $\beta\colon M\to F(M)$ is an $R$-module homomorphism, and the structure isomorphism $\theta$ is induced by the vertical maps in the diagram, see~\cite[Definition~2.1]{Lyubeznik:Crelle}. When $M$ is graded and $\beta$ is degree-preserving, we say that the $\calF$-module $\calM$ is \emph{graded $\calF$-finite}.

The map $\beta\colon M\to F(M)$ above is a \emph{generating morphism} of $\calM$. If $\beta$ is injective, we say that $M$ is a \emph{root} of $\calM$, and that $\beta$ is a \emph{root morphism}. The image of $M$ in $\calM$ will also be called a root of $\calM$. A \emph{minimal root} of $\calM$ is a root $M$ such that no other root of~$\calM$ is contained in~$M$. The minimal root is unique, see~\cite[Theorem~3.5]{Lyubeznik:Crelle}. If $\calM$ is a graded $\calF$-finite module, then its minimal root $M$ is graded, and $\beta\colon M\to F(M)$ is degree-preserving; we say $\beta$ is the \emph{minimal root morphism}~of~$\calM$.

A basic result in the theory of $\calF$-modules says that an $\calF$-finite module $\calM$ has finite length in the category of $\calF$-modules. This means, in particular, that every filtration of $\calM$ in the category of $\calF$-modules can be completed to a maximal filtration
\[
0=\calM_0\subset \calM_1\subset \dots\subset \calM_\ell=\calM\,.
\]
Every maximal filtration has the same length $\ell$, which is defined to be the \emph{length} of $\calM$. The set of the composition factors
\[
\{\calM_1/\calM_0,\dots,\calM_{\ell}/\calM_{\ell-1}\}
\]
depends only on $\calM$, and not on the maximal filtration. It follows from Proposition~\ref{proposition:subcategory} that the composition factors of a graded $\calF$-finite module are all graded.

Set $\frakm$ to be the homogeneous maximal ideal of $R$, and $^*E$ to be the injective hull of~$R/\frakm$ in the category of graded $R$-modules. Shifting the grading by $n$, one has a degree-preserving isomorphism
\[
^*E(n)\ \cong\ H^n_\frakm(R)\,,
\]
see, for example, \cite[Theorem~1.2.7]{GW}. Set
\[
\grD(-)\ =\ \Hom_R(-,\ H^n_\frakm(R))\,,
\]
which is the graded Matlis duality functor; this is a contravariant exact functor. If $M$ is a graded $R$-module that is cofinite (respectively, finitely generated), then $\grD(M)$ is graded and finitely generated (respectively, cofinite). For a graded module $M$ that is cofinite or finitely generated, one has
\[
\grD(\grD(M))\ =\ M\,,
\]
see~\cite[Theorem~1.2.10]{GW}; in particular, there is a one-to-one correspondence between graded submodules of $M$ and graded quotients of $\grD(M)$, namely, an inclusion $N\to M$ corresponds to a surjection $\grD(M)\to\grD(N)$.

The following is a version of \cite[Lemma~4.1]{Lyubeznik:Crelle}; the proof is similar when~$M$ is cofinite, and is readily adapted to the case where $M$ is a finitely generated $R$-module.

\begin{lemma}
\label{lemma:tau}
Let $M$ be a graded $R$-module that is either cofinite or finitely generated. Then there is an $R$-module isomorphism
\[
\grtau\colon\grD(F(M))\to F(\grD(M))
\]
that is degree-preserving, and functorial in $M$.
\end{lemma}

\subsection*{The functor $\calH(-)$}

We set $R\{f\}$ to be the ring extension of $R$ generated by an element~$f$ subject to the relations $fr=r^pf$ for each $r\in R$. By an $R\{f\}$-module we mean a left~$R\{f\}$-module. Thus, an $R\{f\}$-module is an $R$-module $M$ equipped with a Frobenius action, i.e.,~with an additive map~$f\colon M\to M$ such that $f(rm)=r^pf(m)$ for each $m\in M$.

By a \emph{graded} $R\{f\}$-module, we mean a graded $R$-module $M$ such that
\[
f\colon{[M]}_d\to{[M]}_{pd}\qquad\text{ for each integer }\ d\,.
\]
It is straightforward to check that the induced $R$-module homomorphism
\[
F(M)=R'\otimes_RM\to M\,,\qquad\text{ where } r'\otimes m\mapsto r'f(m)\,,
\]
is degree-preserving i.e., it is a morphism in the category of graded $R$-modules. Applying the graded Matlis duality functor $\grD$ to this morphism, the induced natural map
\[
\grD(M)\to{\grD(F(M))}
\]
is degree-preserving. Following this map with $^*\tau$ produces the natural map
\[
\beta_M\colon\grD(M)\to F(\grD(M))\,,
\]
that, again, is degree-preserving. If $M$ is cofinite, then $\grD(M)$ is finitely generated, and we set $\calH(M)$ to be the $\calF$-finite module with generating morphism $\beta_M$. As $\beta_M$ is degree-preserving, the module $\calH(M)$ is graded. Thus, $\calH(-)$ is a functor from the category of graded cofinite $R\{f\}$-modules to the category of graded $\calF$-finite modules.

Let $M$ be an $R\{f\}$-module. An element $m$ of $M$ that is annihilated by some power of~$f$ is said to be \emph{nilpotent}; the module $M$ is \emph{nilpotent} if $f^e(M)=0$ for some~$e$. The set of nilpotent elements of $M$ is an $R\{f\}$-submodule of~$M$, this is the nilpotent part of $M$, denoted $M_\nil$. The \emph{reduced} $R\{f\}$-module
\[
M_\red\ =\ M/M_\nil
\]
has no nonzero nilpotent elements. Set $M^{f^e}$ to be the $R$-submodule generated by the set~$f^e(M)$. We use $M^\st$ to denote the intersection of the descending chain
\[
M\ \supseteq\ M^f\ \supseteq\ M^{f^2}\ \supseteq\ \cdots\,.
\]
Each $M^{f^e}$ is an $R\{f\}$-module, hence so is $M^\st$. It is straightforward to verify that
\[
(M_\red)^\st\ =\ (M^\st)_\red\,,
\]
and we denote this $R\{f\}$-module by $M^\st_\red$. If $M$ is a graded $R\{f\}$-module, then so are the modules $M_\red$, $M^\st$, and $M^\st_\red$. The following is a graded version of \cite[Theorem~4.2]{Lyubeznik:Crelle}:

\begin{theorem}
Consider the functor $\calH(-)$ from the category of graded cofinite $R\{f\}$-modules to the category of graded $\calF$-finite modules. Then:
\begin{enumerate}[\quad\rm(1)]
\item The functor $\calH(-)$ is contravariant, additive, and exact.
\item $\calH(M)=0$ if and only if $M$ is nilpotent.
\item The minimal root morphism of $\calH(M)$ is
\[
\beta_{M^\st_\red}\colon\grD(M^\st_\red)\to F(\grD(M^\st_\red))\,.
\]
\item $\calH(M)$ is isomorphic to $\calH(M')$ in the category of $\calF$-modules if and only if $M^\st_\red$ is isomorphic to $(M')^\st_\red$ in the category of $R\{f\}$-modules.
\end{enumerate}
\end{theorem}

The proofs of assertions (1) and (2) are, aside from minor modifications, the same as those of \cite[Theorem~4.2~(i), (ii)]{Lyubeznik:Crelle}, while the proofs of (3) and (4) require the following lemma that is a graded analogue of \cite[Lemma~4.3]{Lyubeznik:Crelle}. We point out that
\[
\beta_M\colon\grD(M)\to F(\grD(M))
\]
is injective if and only if $M^\st=M$, see \cite[page~105,~lines 3--6]{Lyubeznik:Crelle}.

\begin{lemma}
\label{lemma:roots}
Let $M$ be a graded cofinite $R\{f\}$-module with $M=M^\st$; it follows that $\beta_M$ is a root morphism of $\calH(M)$. Let $N$ be a graded $R$-submodule of $\grD(M)$.
\begin{enumerate}[\quad\rm(1)]
\item $N$ is a root of an $\calF$-submodule $\calN$ of $\calH(M)$ if and only if $N=\grD(M'')$, where $M''$ is a homomorphic image of $M$ in the category of $R\{f\}$-modules; in this case, $\beta_{M''}$ is a root morphism of $\calN$.

\item $N$ is a root of $\calH(M)$ if and only if $N=\grD(M/M')$, where $M'$ is a nilpotent $R\{f\}$-submodule of $M$; in this case, $\beta_{M/M'}$ is a root morphism of $\calH(M/M')$.

\item $N$ is the minimal root of $\calH(M)$ if and only if $N=\grD(M_\red)$; in this case, the morphism $\beta_{M_\red}$ is the minimal root morphism of $\calH(M)$.
\end{enumerate}
\end{lemma}

The proof of the lemma parallels that of \cite[Lemma~4.3]{Lyubeznik:Crelle}. 

\begin{proposition}
\label{proposition:H:onto}
The functor $\calH(-)$ from the category of graded cofinite $R\{f\}$-modules to the category of graded $\calF$-finite modules is surjective. 
\end{proposition}

\begin{proof}
Let $\beta\colon M\to F(M)$ be a generating morphism for a graded $\calF$-finite module $\calM$. Using Lemma~\ref{lemma:tau}, we have an $R$-module homomorphism $\gamma$ which is the composition
\[
\CD
R'\otimes_R\grD(M)@=F(\grD(M))@>{\grtau^{-1}}>>\grD(F(M))@>\grD(\beta)>>\grD(M)\,.
\endCD
\]
We define an additive map $f\colon\grD(M)\to\grD(M)$ by $f(\eta)=\gamma(1\otimes\eta)$. Note that
\[
f(r\,\eta)\ =\ \gamma(1\otimes r\,\eta)\ =\ \gamma(r^p\otimes\eta)\ =\ r^p\gamma(1\otimes\eta)\ =\ r^pf(\eta)\,,
\]
i.e., $\grD(M)$ has a natural $R\{f\}$-module structure. Observe that $\calH(\grD(M))=\calM$.
\end{proof}

\begin{proposition}
\label{proposition:fundamental:example}
Let $I$ be a homogeneous ideal of $R$. Then
\[
\calH(H^{n-k}_\frakm(R/I))\ \cong\ H^k_I(R)\,.
\]
\end{proposition}

The proof mirrors that of \cite[Example~4.8]{Lyubeznik:Crelle}; one replaces local duality by graded local duality, which says that if $M$ is a finitely generated graded $R$-module, then there is a natural functorial degree-preserving isomorphism
\[
\grD(H^{n-k}_\frakm(M))\ \cong\ \Ext^k_R(M,R(-n))\,,
\]
see~\cite[Proposition~2.1.6]{GW}; note that $R(-n)$ is the graded canonical module of $R$.

We now prove our main theorem on graded $\calF$-modules:

\begin{theorem}
\label{theorem:equivalence}
Let $M$ be a graded cofinite $R\{f\}$-module. Then the following are equivalent:
\begin{enumerate}[\quad\rm(1)]
\item Among the composition factors of the Eulerian $\calD$-module $\xi(\calH(M))$, there is at least one composition factor with support $\{\frakm\}$.

\item Among the composition factors of the graded $\calF$-finite module $\calH(M)$, there is at least one composition factor with support $\{\frakm\}$.

\item There exists an $\calF$-submodule $\calM$ of \ $\calH(M)$ such that every composition factor of~$\calM$ has support bigger than $\{\frakm\}$, and \ $\calH(M)/\calM$ has support $\{\frakm\}$.

\item The action of the Frobenius $f$ on ${[M]}_0$, the degree zero part of $M$, is not nilpotent.
\end{enumerate}
\end{theorem}

\begin{proof}
Without loss of generality, we assume that $M=M^\st_\red$ is reduced; each of the statements is unaffected by replacing $M$ with $M^\st_\red$.

By \cite[Theorem~5.6]{Lyubeznik:Crelle}, if $\calM$ is an $\calF$-finite module that is simple in the category of $\calF$-modules, then, in the category of $\calD$-modules, $\xi(\calM)$ is the direct sum of finitely many simple $\calD$-modules, say $\xi(\calM)\cong\oplus_i\calN_i$, where each $\calN_i$ is a simple $\calD$-module.

If $\calM$ is any $\calF$-finite module, then the composition factors of $\xi(\calM)$ in the category of $\calD$-modules are the modules $\calN_i$ appearing in the direct sum decomposition of the modules~$\xi(\calM')$, where $\calM'$ runs through the composition factors of $\calM$ in the category of $\calF$-modules. By~\cite[Theorem~2.12]{Lyubeznik:Crelle}, each simple $\calF$-module $\calM'$ has a unique associated prime, which must then be the unique associated prime of each $\calN_i$ appearing in the direct sum decomposition of $\xi(\calM')$ in the category of $\calD$-modules. Thus, $\xi(\calM)$ has a composition factor with support $\{\frakm\}$ if and only if $\calM$ has a composition factor with support $\{\frakm\}$. This proves the equivalence of (1) and (2).

Note that $\grD(M)\to\calH(M)$ is injective since $M=M^\st_\red$; we think of $\grD(M)$ as a submodule of $\calH(M)$ via this map. The map $\grD(M)\to F(\grD(M))$ is the minimal root morphism of $\calH(M)$. Let
\[
0=\calM_0\subset\calM_1\subset\dots\subset\calM_{\ell}=\calH(M)
\]
be a maximal filtration of $\calH(M)$ in the category of $\calF$-finite modules. Set $N_i$ to be the module $\calM_i\cap\grD(M)$. Then $N_i$ is a root of $\calM_i$. The surjection $M\to\grD(N_i)$ is an $R\{f\}$-module map. We denote the kernel of this surjection by $M_i$; this is an $R\{f\}$-submodule of~$M$; thus, there exists a chain of graded $R\{f\}$-submodules
\[
M=M_0\supset M_1\supset\cdots\supset M_{\ell}=0\,,
\]
such that the natural map
\[
\beta_{M_i/M_{i+1}}\colon\grD(M_i/M_{i+1})\to F(\grD(M_i/M_{i+1}))
\]
is a generating morphism of $\calM_{i+1}/\calM_i$. 

Suppose $\calM_{i+1}/\calM_i$ has support $\{\frakm\}$. Since $\grD((M_i/M_{i+1})^\st)$ is a root of $\calM_{i+1}/\calM_i$, hence isomorphic to a submodule of $\calM_{i+1}/\calM_i$, it has support $\{\frakm\}$. As~$\grD((M_i/M_{i+1})^\st)$ is finitely generated, it has finite length as an $R$-module. It follows that $(M_i/M_{i+1})^\st_\red$ is concentrated in degree zero: indeed, if $m$ is a nonzero element of degree $d\neq0$, then, for each $e$, the element $f^e(m)$ is nonzero of degree $dp^e$, contradicting the finite length.

Since $\calM_{i+1}/\calM_i$ is nonzero, the action of the Frobenius $f$ on $(M_i/M_{i+1})^\st_\red$ cannot be nilpotent. But~$(M_i/M_{i+1})^\st_\red$ is a subquotient of ${[M]}_0$, hence the action of $f$ on ${[M]}_0$ is not nilpotent. This proves that (2) implies (4).

Assuming (4) holds, set $M'={[M]}_{\ge0}$, which is the $R$-submodule of $M$ generated by the homogeneous elements of nonnegative degree. This is then an $R\{f\}$-submodule of $M$, and one has an exact sequence in the category of graded $R\{f\}$-modules,
\[
\CD
0@>>>M'@>>>M@>>>M/M'@>>>0\,.
\endCD
\]
This yields the exact sequence in the category of graded $\calF$-modules,
\[
\CD
0@>>>\calH(M/M')@>>>\calH(M)@>>>\calH(M')@>>>0\,.
\endCD
\]
Since $M'$ has finite length and a non-nilpotent Frobenius action, the module $\calH(M')$ is nonzero with support $\{\frakm\}$. Since ${[M/M']}_0=0$, it follows from the fact that (2) implies~(4) that all the composition factors of $\calH(M/M')$ have support bigger than $\{\frakm\}$. This proves that (4) implies (3), which, in turn, trivially implies (2).
\end{proof}

By Hochster~\cite[Theorem~3.1]{Hochster:Fmodules}, the category of $\calF$-modules has enough injectives. However, since $H^n_\frakm(R)$ is typically not an injective object in the category of $\calF$-finite modules, see \cite[Example~4.8]{Ma}, the following corollary is very unexpected:

\begin{corollary}
\label{corollary:injective}
Let $R$ be a standard graded polynomial ring of dimension $n$ over a separably closed field. Then the $R$-module $H^n_\frakm(R)$, with its natural $\calF$-module structure, is an injective object in the category of graded $\calF$-finite modules. 
\end{corollary}

\begin{proof}
Let $\calM$ be a graded $\calF$-finite module with $H^n_\frakm(R)$ as an $\calF$-submodule; it suffices to show that $H^n_\frakm(R)\subset\calM$ splits in the category of graded $\calF$-modules. The module $H^n_\frakm(R)$ is a composition factor of $\calM$ with support $\{\frakm\}$; we first reduce to the case where $\calM$ has support precisely $\{\frakm\}$ as follows.

By Theorem~\ref{theorem:equivalence}, there exists a surjection $\phi\colon\calM\to\calN$ of graded $\calF$-modules such that each composition factor of $\ker\phi$ has support bigger than $\{\frakm\}$, and $\calN$ has support~$\{\frakm\}$. Since $H^n_\frakm(R)$ is a simple $\calF$-module that is not in $\ker\phi$, it maps to an isomorphic copy that is an $\calF$-submodule of $\calN$. Assuming that there is a splitting $\calN=\phi(H^n_\frakm(R))\oplus\calN'$ in the category of graded $\calF$-modules, the composition
\[
\CD
\calM @>\phi>> \calN @= \phi(H^n_\frakm(R))\oplus\calN' @>\pi_1>> \phi(H^n_\frakm(R)) @>\phi^{-1}>> H^n_\frakm(R)\,,
\endCD
\]
where $\pi_1$ is the projection to the first component, provides a splitting of $H^n_\frakm(R)\subset\calM$ in the category of graded $\calF$-modules.

We may thus assume that $\calM$ is a graded $\calF$-finite module with support $\{\frakm\}$; we need to show that $H^n_\frakm(R)\subset\calM$ splits in the category of $\calF$-modules. Take $\grD(M)$ to be the minimal root of $\calM$; then $M$ is a graded $R\{f\}$-module by Proposition~\ref{proposition:H:onto}. Note that $M_\red=M$, and $M$ has finite length as an $R$-module. Since homogeneous elements of $M$ of nonzero degree are necessarily nilpotent, it follows that $M$ is concentrated in degree $0$. Thus, $M$ is annihilated by $\frakm$, and is a finite $\FF\{f\}$ module, where $\FF$ is viewed as the residue field $R/\frakm$.

Since $H^n_\frakm(R)$ is an $\calF$-submodule of $\calM$, there exists an $\FF\{f\}$-module homomorphic image $N$ of $M$ such that 
\[
\grD(N)=\grD(M)\cap H^n_\frakm(R)\,.
\]
By the following lemma, the surjection $M\to N$ splits in the category of $\FF\{f\}$-modules. Applying $\calH$, the inclusion $H^n_\frakm(R)\subset\calM$ splits in the category of graded $\calF$-modules.
\end{proof}

\begin{remark}
For $R$ as in Corollary~\ref{corollary:injective}, we do not know whether $H^n_\frakm(R)$ is injective in the category of graded $\calF$-modules.
\end{remark}

\begin{lemma}
Let $\FF$ be a separably closed field of positive characteristic. Then every exact sequence of $\FF\{f\}$-modules
\[
\CD
0@>>>L@>>>M@>>>N@>>>0\,,
\endCD
\]
where $L$, $M$, $N$, are $\FF$-vector spaces of finite rank, splits in the category of $\FF\{f\}$-modules.
\end{lemma}

\begin{proof}
We identify $L$ with its image in $M$, and $N$ with $M/L$. Using \cite[Theorem~4.2]{Hochster:Fmodules}, choose a basis $e_1,\dots,e_\ell$ for $L$ such that $f(e_i)=e_i$ for each $i$; when $\FF$ is algebraically closed, this also follows from \cite[page~233]{Dieudonne}. Similarly, $N$ has a basis $v_1,\dots,v_n$ with $f(v_j)=v_j$ for each $j$. It suffices to prove that each~$v_j$ lifts to an element $w_j\in M$ with $f(w_j)=w_j$. 

Set $v=v_j$ and let $\tilde{v}$ in $M$ be a lift of $v$. Since $f(v)=v$, it follows that $f(\tilde{v})-\tilde{v}$ is an element of $L$. Thus, there exist elements $c_i\in\FF$ with 
\[
f(\tilde{v})-\tilde{v}\ =\ \sum_{i=1}^\ell c_ie_i\,.
\]
For each $i$, the separable equation
\[
T^p-T+c_i\ =\ 0
\]
has a root $t_i$ in $\FF$. Setting
\[
w\ =\ \tilde{v}+\sum_{i=1}^\ell t_ie_i\,,
\]
it is readily seen that $f(w)=w$.
\end{proof}

The following example of Ma shows that the corollary does not hold over arbitrary fields; more generally, Ma computes the relevant $\Ext$ groups in \cite[Theorem~4.5]{Ma}.

\begin{example}
We consider $\calF$-modules over the field $\FF=\FF_p$. Take $\calM$ to be $\FF\oplus\FF$ with structure morphism
\[
\theta_{\!\calM}\colon\calM\to F(\calM)\,,\qquad (a,b)\mapsto\big(a\otimes1,(a+b)\otimes1\big)\,.
\]
Then $\FF$, with structure morphism
\[
\theta_{\,\FF}\colon\FF\to F(\FF)\,,\qquad b\mapsto b\otimes 1\,,
\]
may be identified with the $\calF$-submodule $0\oplus\FF$ of $\calM$. We claim that the inclusion $\FF\subset\calM$ does not split in the category of $\calF$-modules. Indeed, a splitting is a map of $\FF$-vector spaces
\[
\phi\colon\calM\to\FF\,,\qquad (a,b)\mapsto a\alpha+b
\]
for some $\alpha$ in $\FF$, such that the following diagram commutes:
\[
\CD
\calM@>\phi>>\FF\\
@V\theta_{\!\calM} VV @VV\theta_{\,\FF} V\\
F(\calM)@>F(\phi)>>F(\FF)\,
\endCD
\]
However, $\theta_{\,\FF}\circ\phi(a,b)=(a\alpha+b)\otimes 1$, whereas
\begin{multline*}
F(\phi)\circ\theta_{\!\calM}(a,b)\ =\ F(\phi)\big(a\otimes1,(a+b)\otimes1\big)\ =\ F(\phi)\big(a\otimes(1,1)+b\otimes(0,1)\big) \\
=\ a\otimes(\alpha+1)+b\otimes1\ =\ (a\alpha^p+a+b)\otimes1\,.
\end{multline*}
Thus, the commutativity forces $\alpha^p+1=\alpha$, which is not possible for $\alpha\in\FF_p$.
\end{example}

\begin{example}
\label{example:Ma}
Let $R=\FF[x_1,\dots,x_n]$, where $n\ge1$ and $\FF$ is an algebraically closed field of characteristic $p>0$. By \cite[Example~4.8]{Ma}, there exists an exact sequence
\[
\CD
0@>>>H^n_\frakm(R)@>>>\calM@>>>R@>>>0
\endCD
\]
that is not split in the category of $\calF$-finite modules. Since $H^n_\frakm(R)$ is an injective object in the category of graded $\calF$-finite modules by Corollary~\ref{corollary:injective}, it follows that $\calM$ is not a graded $\calF$-module; thus, the category of graded $\calF$-modules is not closed---as a subcategory of the category of $\calF$-modules---under extensions.
\end{example}

We record another consequence of Theorem~\ref{theorem:equivalence}:

\begin{corollary}
If $\calM'$ and $\calM''$ are graded $\calF$-finite modules such that $\calM'$ has support~$\{\frakm\}$ and $\calM''$ has no composition factor with support $\{\frakm\}$, then every extension
\[
\CD
0@>>>\calM'@>>>\calM@>>>\calM''@>>>0
\endCD
\]
in the category of graded $\calF$-modules is split.
\end{corollary}

\begin{proof}
By Theorem~\ref{theorem:equivalence}, there exists an $\calF$-module surjection $\calM\to\calM_1$ where $\calM_1$ is an $\calF$-module with support $\{\frakm\}$, and the kernel of this surjection has no composition factor with support $\{\frakm\}$. Restricting to $\calM'$, the surjection induces an isomorphism $\calM'\to\calM_1$. Thus, we have an $\calF$-module splitting $\calM\to\calM'$.
\end{proof}

Applying Theorem~\ref{theorem:equivalence} to Proposition~\ref{proposition:fundamental:example}, we obtain the following theorem:

\begin{theorem}
\label{theorem:quotient:graded}
Let $R$ be a standard graded polynomial ring, where ${[R]}_0$ is a field of prime characteristic. Let $\frakm$ be the homogeneous maximal ideal of $R$, and $I$ an arbitrary homogeneous ideal. For each nonnegative integer $k$, the following are equivalent:

\begin{enumerate}[\quad\rm(1)]
\item Among the composition factors of the Eulerian $\calD$-module $\xi(H^k_I(R))$, there is at least one composition factor with support $\{\frakm\}$.

\item Among the composition factors of the graded $\calF$-finite module $H^k_I(R)$, there is at least one composition factor with support $\{\frakm\}$.

\item $H^k_I(R)$ has a graded $\calF$-module homomorphic image with support $\{\frakm\}$.

\item The natural Frobenius action on ${[H^{\dim R-k}_\frakm(R/I)]}_0$ is not nilpotent.
\end{enumerate}
\end{theorem}

\begin{example}
Consider the polynomial ring $R=\FF_p[x_1,\dots,x_6]$, where $p$ is a prime integer. Let $\frakm$ denote the homogeneous maximal ideal of $R$, and set $I$ to be the ideal generated by
\[
x_1x_2x_3\,,\ \ x_1x_2x_4\,,\ \ x_1x_3x_5\,,\ \ x_1x_4x_6\,,\ \ x_1x_5x_6\,,\ \ x_2x_3x_6\,,\
x_2x_4x_5\,,\ \ x_2x_5x_6\,,\ \ x_3x_4x_5\,,\ \ x_3x_4x_6\,;
\]
this is the Stanley-Reisner ideal for a triangulation of the real projective plane $\RR\PP^2$ as in~\cite[Example~5.2]{SW:Crelle}. The ideal $I$ height $3$. We claim that $H^3_I(R)$ has a graded $\calF$-module homomorphic image with support $\{\frakm\}$ if and only if $p=2$.

For each $k\ge1$, one has
\[
{[H^{k+1}_\frakm(R/I)]}_0\ =\ \Hsing^k(\RR\PP^2\,;\,\ZZ/p\ZZ)\,,
\]
by Hochster's formula, see, for example, \cite[Section~5.3]{Bruns-Herzog}. Using this,
\[
{[H^3_\frakm(R/I)]}_0\ =\ 
\begin{cases}
\ZZ/2\ZZ&\text{ if }\ p=2\,,\\
0&\text{ if }\ p>2\,.
\end{cases}
\]
The ring $R/I$ is $F$-pure since $I$ is a square-free monomial ideal; when $p=2$, the Frobenius action on ${[H^3_\frakm(R/I)]}_0$ is thus injective. The claim now follows from Theorem~\ref{theorem:quotient:graded}.
\end{example}

\begin{corollary}
\label{corollary:p:torsion}
Let $R=\ZZ[x_1,\dots,x_n]$ be a polynomial ring with the $\NN$-grading ${[R]}_0=\ZZ$ and~$\deg x_i=1$ for each~$i$. Let $I$ be a homogeneous ideal, $p$ a prime integer, and $k$ a nonnegative integer. Suppose that the Frobenius action on
\[
{\big[H^{n-k}_{(x_1,\dots,x_n)}(R/(I+pR))\big]}_0
\]
is nilpotent, and that the multiplication by $p$ map
\[
\CD
{H^{k+1}_I(R)}_{x_i}@>{\cdot p}>>{H^{k+1}_I(R)}_{x_i}
\endCD
\]
is injective for each $i$. Then the multiplication by $p$ map on $H^{k+1}_I(R)$ is injective.
\end{corollary}

\begin{proof}
The ring $\calD_\ZZ(R)$ of differential operators on $R$ is a free~$R$-module with basis
\[
\del_1^{[t_1]}\ \cdots\ \del_n^{[t_n]}\qquad\text{ for } \ (t_1,\dots,t_n)\in\NN^n\,,
\]
see~\cite[Th\'eor\`eme~16.11.2]{EGA4}. Multiplication by $p$ on $R$ induces
\[
\CD
@>>>H^k_I(R)@>>>H^k_I(R/pR)@>{\delta}>>H^{k+1}_I(R)@>{\cdot p}>>H^{k+1}_I(R)@>>>\,,
\endCD
\]
which is an exact sequence of $\calD_\ZZ(R)$-modules. Specifically, the kernel of multiplication by~$p$ on $H^{k+1}_I(R)$ is a $\calD_\ZZ(R)$-module; since it is annihilated by~$p$, it is also a module over
\[
\calD_\ZZ(R)/p\calD_\ZZ(R)\ =\ \calD_{\,\FF_p}(R/pR)\,.
\]
If this kernel is nonzero, then it is a homomorphic image of $H^k_I(R/pR)$ in the category of Eulerian $\calD_{\,\FF_p}(R/pR)$-modules, supported precisely at the homogeneous maximal ideal~$\frakm$ of $R/pR$. But this is not possible, since the $\calD_{\,\FF_p}(R/pR)$-module $H^k_I(R/pR)$ has no composition factor with support $\{\frakm\}$ by Theorem~\ref{theorem:quotient:graded}.
\end{proof}

\begin{example}
\label{example:ep1}
Let $E$ be an elliptic curve in $\PP^2_\QQ$, and consider the Segre embedding of the product $E\times\PP^1_\QQ$ in $\PP_\QQ^5$. Set $R=\ZZ[x_1,\dots,x_6]$, and let $I\subset R$ be an ideal such that $(R/I)\otimes_\ZZ\QQ$ is the homogeneous coordinate ring of the embedding. For all but finitely many primes~$p$, the reduction of $E$ modulo $p$ is an elliptic curve that we denote by $E_p$. By Serre~\cite{Serre} and Elkies~\cite{Elkies} respectively, there exist infinitely many prime integers $p$ such that $E_p$ is ordinary, and infinitely many such that $E_p$ is supersingular.

Take a prime $p$ for which $E_p$ is an elliptic curve; then $(R/I)\otimes_\ZZ\FF_p$ is a homogeneous coordinate ring for $E_p\times\PP^1_{\FF_p}$. Using the K\"unneth formula, one obtains
\[
H^2_\frakm(R/(I+pR))
=\ H^1(E_p,\calO_{E_p})\otimes H^0(\PP^1_{\FF_p},\calO_{\PP^1_{\FF_p}})\,.
\]
Hence, the Frobenius action on the rank one $\FF_p$-vector space $H^2_\frakm(R/(I+pR))$ may be identified with the map
\[
\CD
H^1(E_p,\calO_{E_p})\otimes H^0(\PP^1_{\FF_p},\calO_{\PP^1_{\FF_p}}) @>f>>
H^1(E_p,\calO_{E_p})\otimes H^0(\PP^1_{\FF_p},\calO_{\PP^1_{\FF_p}}),
\endCD
\]
which is zero when $E_p$ is supersingular, and nonzero when $E_p$ is ordinary. It follows that the module $H^2_\frakm(R/(I+pR))^\st$ is zero when $E_p$ is supersingular, and nonzero when it is ordinary. By \cite[page~75]{HS} or \cite[Theorem~3.1]{Lyubeznik:Compositio}, the same holds for $H^4_I(R/pR)$, implying that the multiplication by $p$ map
\[
\CD
H^4_I(R)@>{\cdot p}>>H^4_I(R)
\endCD
\]
is surjective for infinitely many prime integers $p$, and also not surjective for infinitely many~$p$; see also~\cite{SW:Contemp}. Corollary~\ref{corollary:p:torsion} implies that the map is \emph{injective} for each $p$ for which $E_p$ is an elliptic curve, since
\[
\left[H^3_\frakm(R/(I+pR))\right]_0\ =\ H^1(E_p,\calO_{E_p})\otimes H^1(\PP^1_{\FF_p},\calO_{\PP^1_{\FF_p}})\ =\ 0\,,
\]
and ${H^4_I(R)}_{x_i}=0$ for each $i$ because the arithmetic rank of $IR_{x_i}$ (defined in Section~\ref{sec:lc}) is $3$; compare with \cite[Example~3.3]{BBLSZ}.
\end{example}

\section{Preliminaries on local cohomology}
\label{sec:lc}

The following theorem enables the calculation of the Bass numbers of certain local cohomology modules in terms of singular cohomology:

\begin{theorem}
\label{theorem:comparison}
Consider the polynomial ring $R=\CC[x_1,\dots,x_n]$. Let $I$ be an ideal of $R$, and~$\frakm$ a maximal ideal. If $k_0$ is a positive integer such that $\Supp H^k_I(R)\subseteq\{\frakm\}$ for each integer $k\ge k_0$. Then, for each $k\ge k_0$, one has an isomorphism of $R$-modules
\[
H^k_I(R)\ \cong\ H^n_\frakm(R)^{\mu_k},
\]
where $\mu_k$ is the $\CC$-rank of the singular cohomology group $\Hsing^{n+k-1}(\CC^n\setminus\Var(I)\,;\,\CC)$.

If $I$ and $\frakm$ are homogeneous with respect to the standard grading on $R$, then the displayed isomorphism is degree-preserving.
\end{theorem}

\begin{proof}
Set $\calD$ to be the Weyl algebra $R\langle \del_1,\dots,\del_n\rangle$, where $\del_j$ denotes partial differentiation with respect to the variable $x_j$; this is the ring of $\CC$-linear differential operators on~$R$. Each~$H^k_I(R)$ is a holonomic $\calD$-module, see for example, \cite[Section~2]{Lyubeznik:Invent} or \cite[Lecture~23]{24hours}. We claim that for each integer $k$ with $k\ge k_0$, the module $H^k_I(R)$ is isomorphic, as a $\calD$-module, to a finite direct sum of copies of the injective hull $E=H^n_\frakm(R)$ of $R/\frakm$ as an $R$-module. This follows from Kashiwara's equivalence, \cite[Proposition~4.3]{Kashiwara}; alternatively, see~\cite[Lemma~(c),~page~208]{Lyubeznik:injective}.

For each $k\ge k_0$, set $\mu_k$ to be the $\CC$-rank of the socle of $H^k_I(R)$; it follows that
\[
H^k_I(R)\ \cong\ E^{\mu_k}\,.
\]

Regard $\del_j$ as the endomorphism of $\calD$ which sends a differential operator $P$ to the composition $\del_j\cdot P$. Then $\del_1,\dots,\del_n$ are commuting endomorphisms of $\calD$. Let $K^\bullet(\bdel;\calD)$ be the Koszul complex on these endomorphisms; this is a complex of right $\calD$-modules. For a left $\calD$-module $M$, set
\[
\dR(M)\ =\ K^\bullet(\bdel;\calD)\otimes_\calD M\,,
\]
which is typically a complex of infinite-dimensional $\CC$-vector spaces. Define $\dR^i(M)$ to be the $i$-th cohomology group of the complex $\dR(M)$. We regard $\dR(-)$ as a functor from the category of $\calD$-modules to the category of complexes of $\CC$-vector spaces. Alternatively, consider the map that is the projection from $X=\Spec R$ to a point; then $\dR(M)$ is the direct image of $M$ under the projection map; see \cite[Section~VI.5]{Borel}.

If $M$ is a holonomic $\calD$-module, then each $\dR^i(M)$ is a $\CC$-vector space of finite rank by~\cite[Theorem~VII.10.1]{Borel}. It is straightforward to verify that
\[
\dR^i(E)\ =\ \begin{cases}0 & \text{ if }i\neq n\,;\\ \CC & \text{ if }i=n\,.\end{cases}
\]
For $k\ge k_0$, it follows that the complex $\dR(H^k_I(R))$ is concentrated in cohomological degree~$n$, and that
\[
\dR^n(H^k_I(R))\ =\ \CC^{\mu_k}\,.
\]

Let $U$ be a Zariski open subset of $X$, and let $U^{an}$ be the corresponding analytic open subset. By the Poincar\'e Lemma, the complex $K^\bullet(\bdel;\calD)\otimes_\calD\calO_{U^{an}}$ is a resolution of the constant sheaf $\CC$ on $U^{an}$. Grothendieck's Comparison Theorem,~\cite[Theorem~1]{derham}, shows that the hypercohomology of the complex $K^\bullet(\bdel;\calD)\otimes_\calD\calO_U$ coincides with the cohomology of the constant sheaf on $U^{an}$, which is the singular cohomology of $U^{an}$.

For an element $g$ of $R$, set $U_g=X\setminus\Var(g)$. Since $U_g$ is affine, and hence $U_g^{an}$ is Stein, the singular cohomology of $U_g^{an}$ is the cohomology of the complex $\dR(R_g)$.

Let $\bsg=g_1,\dots,g_m$ be generators of $I$, and consider the complex of left $\calD$-modules
\[
\Cdot(\bsg;R)\ : \qquad 0\to\bigoplus_iR_{g_i}\to\bigoplus_{i<j}R_{g_ig_j}\to\cdots\to R_{g_1\cdots g_m}\to 0\,,
\]
that is supported in cohomological degrees $0,\dots,m$. For each $p\ge 1$, this complex has cohomology $H^p(\Cdot(\bsg;R))=H^{p+1}_I(R)$.

The sets $U_{g_i}=X\setminus\Var(g_i)$ form an affine open cover for $U=X\setminus\Var(I)$, so the double complex $Q^{\bullet,\bullet}$ with
\[
Q^{p,q}\ =\ K^q(\bdel;\calD)\otimes_\calD C^p(\bsg;R)
\]
is a local trivialization of $\dR(\calO_U)$. It follows that the cohomology of the total complex of $Q^{\bullet,\bullet}$ is the singular cohomology of $U^{an}$, see \cite[Theorem~8.9]{BottTu}, and the surrounding discussion. Consider the spectral sequence associated to $Q^{\bullet,\bullet}$, with the differentials
\[
E^{p,q}_r\to E^{p-r+1,q+r}_r\,.
\]
Taking cohomology along the rows, one obtains the $E_1$ page of the spectral sequence, where the $q$-th column is 
$\dR(H^p(\Cdot(\bsg;R)))$. Thus,
\begin{align*}
E^{p,q}_2\ &=\ \dR^q(H^p(\Cdot(\bsg;R)))\\
&=\ \dR^q(H^{p+1}_I(R))\qquad\text{ for }p\ge1\,.\\
\end{align*}

Suppose that $p\ge\max\{1,k_0-1\}$. Then $E^{p,q}_2=\dR^q(E^{\mu_{p+1}})$, which is zero for $q\neq n$. It follows that the differentials to and from $E^{p,q}_2$ are zero, and so $E^{p,q}_\infty=E^{p,q}_2$. In particular,
\[
\Hsing^{p+n}(U^{an})\ =\ E^{p,n}_\infty=\CC^{\mu_{p+1}}\qquad\text{ for }p\ge\max\{1,k_0-1\}\,.
\]
This proves the isomorphism asserted in the theorem for $k\ge\max\{2,k_0\}$.

It remains to consider the case where $H^k_I(R)$ is $\frakm$-torsion for each $k\ge 1$. If there exists a minimal prime $\frakp$ of $I$ with $\frakp\neq\frakm$, then $H^k_\frakp(R_\frakp)=0$ for each $k\ge 1$, which forces $\frakp=0$ and thus $I=0$; the theorem holds trivially in this case. Lastly, we have the case where $I$ has radical $\frakm$; without loss of generality, $I=\frakm$. Then the only nonvanishing module $H^\bullet_\frakm(R)$ is~$H^n_\frakm(R)=E$; since $\CC^n\setminus\Var(\frakm)$ is homotopic to the real sphere $\SS^{2n-1}$, we have
\[
\Hsing^{n+k-1}(\CC^n\setminus\Var(\frakm)\,;\,\CC)\ =\ 
\begin{cases}0&\text{ if }1\le k\le n-1\\ \CC&\text{ if }k=n\,.\end{cases}
\]

If $I$ and $\frakm$ are homogeneous, then $H^k_I(R)$ and $H^n_\frakm(R)$ are Eulerian graded $\calD$-modules and the isomorphism $H^k_I(R)\cong H^n_\frakm(R)^{\mu_k}$ is degree-preserving by \cite[Theorem 1.1]{Ma-Zhang}.
\end{proof}

\subsection*{Arithmetic rank}

The \emph{arithmetic rank} of an ideal $I$ of a ring $A$, denoted $\ara I$, is the least integer $k$ such that
\[
\rad I\ =\ \rad (g_1,\dots,g_k)A
\]
for elements $g_1,\dots,g_k$ of $A$. It is readily seen that $H^i_I(A)=0$ for each $i>\ara I$. The corresponding result for singular cohomology is the following, see~\cite[Lemma~3]{Bruns-Schwanzl}:

\begin{lemma}
\label{lemma:ara}
Let $W\subseteq\tilde{W}$ be affine varieties over $\CC$, such that $\tilde{W}\setminus W$ is nonsingular of pure dimension $d$. If there exist~$k$ polynomials $f_1,\dots,f_k$ with
\[
W\ =\ \tilde{W}\cap\Var(f_1,\dots,f_k)\,,
\]
then
\[
\Hsing^{d+i}(\tilde{W}\setminus W\,;\,\CC)=0\qquad\text{ for each }i\ge k\,.
\]
\end{lemma}

\begin{lemma}
\label{lemma:basechange}
Let $B\to A$ be a homomorphism of commutative rings; let $I$ be an ideal of~$B$. If $I$ can be generated up to radical by $k$ elements, then
\[
H^k_I(B)\otimes_BA\ \cong\ H^k_{IA}(A)\,.
\]
\end{lemma}

\begin{proof}
Let $b_1,\dots,b_k$ be elements of $B$ that generate $I$ up to radical. Computing $H^k_I(B)$ using a \v Cech complex on the $b_i$, one obtains $H^k_I(B)$ as the cokernel of the homomorphism
\[
\CD
\sum_iB_{b_1\cdots\hat{b}_i\cdots b_k}@>>>B_{b_1\cdots b_k}\,.
\endCD
\]
Since the functor $-\otimes_BA$ is right-exact, $H^k_I(B)\otimes_BA$ is isomorphic to the cokernel of
\[
\CD
\sum_iA_{b_1\cdots\hat{b}_i\cdots b_k}@>>>A_{b_1\cdots b_k}\,,
\endCD
\]
which is the local cohomology module $H^k_{IA}(A)$.
\end{proof}

\subsection*{The $a$-invariant}

Let $A$ be an $\NN$-graded ring such that ${[A]}_0$ is a field; let $\frakm$ be the homogeneous maximal ideal of $A$. Following \cite[Definition~3.1.4]{GW}, the \emph{$a$-invariant} of $A$, denoted~$a(A)$, is the largest integer $k$ such that
\[
{[H^{\dim A}_\frakm(A)]}_k\ \neq\ 0\,.
\]

The following lemma is taken from~\cite[Discussion~7.4]{HH:JAG}:
\begin{lemma}
\label{lemma:ainv}
Let $A$ be an $\NN$-graded ring with ${[A]}_0=\ZZ$, that is finitely generated as an algebra over ${[A]}_0$. Assume, moreover, that $A$ is a free $\ZZ$-module. Let $p$ be a prime integer. If the rings $A/pA$ and $A\otimes_\ZZ\QQ$ are Cohen-Macaulay, then
\[
a(A/pA)\ =\ a(A\otimes_\ZZ\QQ)\,.
\]
\end{lemma}

\begin{proof}
The freeness hypothesis implies that for each integer $n$, one has
\[
\rank_{\FF_p}{[A/pA]}_n\ =\ \rank_\ZZ{[A]}_n\ =\ \rank_\QQ{[A\otimes_\ZZ\QQ]}_n\,,
\]
so the rings $A/pA$ and $A\otimes_\ZZ\QQ$ have the same Hilbert-Poincar\'e series; the rings are Cohen-Macaulay, so the Hilbert-Poincar\'e series determines the $a$-invariant.
\end{proof}

Suppose $A$ is an $\NN$-graded normal domain that is finitely generated over a field ${[A]}_0$ of characteristic zero. Consider a desingularization $\phi\colon Z\to\Spec A$, i.e., a proper birational morphism with $Z$ a nonsingular variety. Then $A$ has \emph{rational singularities} if
\[
R^i\phi_*\calO_Z\ =\ 0\qquad\text{ for each }\ i\ge 1\,;
\]
the vanishing is independent of $\phi$. By Flenner~\cite{Flenner} or Watanabe~\cite{KW:kstar}, $a(A)$ is negative whenever $A$ has rational singularities. By Boutot's theorem~\cite{Boutot}, direct summands of rings with rational singularities have rational singularities; specifically, if $A$ is the ring of invariants of a linearly reductive group acting linearly on a polynomial ring, then $a(A)<0$.

\subsection*{Associated primes of local cohomology}

Huneke~\cite[Problem~4]{Huneke:Sundance} asked whether local cohomology modules of Noetherian rings have finitely many associated prime ideals. A counterexample was given by Singh \cite{Singh:MRL}, see Example~\ref{example:hypersurface} below, by constructing $p$-torsion elements for each prime integer $p$. The same paper disproved a conjecture of Hochster about $p$-torsion elements in $H^3_{I_2}(\ZZ[X])$, where $X$ is a $2\times 3$ matrix of indeterminates, and motivated our study of $p$-torsion in local cohomology modules $H^k_{I_t}(\ZZ[X])$, for~$I_t$ a determinantal ideal; this is completely settled by Theorem~\ref{theorem:intro:rational} of the present paper.

\begin{example}
\label{example:hypersurface}
Let $A$ be the hypersurface $\ZZ[u,v,w,x,y,z]/(ux+vy+wz)$, and let $\fraka$ be the ideal~$(x,y,z)$. By \cite{Singh:MRL} the module $H^3_\fraka(A)$ has $p$-torsion for each prime integer $p$, equivalently, $H^3_\fraka(A)$, viewed as an Abelian group, contains a copy of $\ZZ/p\ZZ$ for each $p$. Chan \cite{Chan} proved that $H^3_\fraka(A)$ contains a copy of each finitely generated Abelian group; moreover, the ring and module in question have a $\ZZ^4$-grading, and Chan shows that any finitely generated Abelian group may be embedded into a single $\ZZ^4$-graded component.
\end{example}

When $R$ is a regular ring, $H^k_\fraka(R)$ is conjectured to have finitely many associated prime ideals, \cite[Remark~3.7]{Lyubeznik:Invent}. This conjecture is now known to be true when $R$ has prime characteristic by Huneke and Sharp~\cite{HS:TAMS}; when $R$ is local or affine of characteristic zero by Lyubeznik~\cite{Lyubeznik:Invent}; when $R$ is an unramified regular local ring of mixed characteristic by~\cite{Lyubeznik:Comm}; and when $R$ is a smooth $\ZZ$-algebra by \cite{BBLSZ}. For rings $R$ of equal characteristic, local cohomology modules $H^k_\fraka(R)$ with infinitely many associated prime ideals were constructed by Katzman~\cite{Katzman}, and subsequently Singh and Swanson, \cite{SS:IMRN}.

A related question is whether, for Noetherian rings $A$, the sets of primes that are minimal in the support of $H_\fraka^k(A)$ is finite, equivalently, whether the support is closed in~$\Spec A$. For positive answers on this, we point the reader towards~\cite{HKM} and the references therein.

\section{Determinantal ideals}
\label{sec:determinantal}

We prove Theorem~\ref{theorem:intro:rational} using the results of the previous sections; we begin with a well-known lemma, see, for example, \cite[Proposition~2.4]{Bruns-Vetter}. We sketch the proof since it is an elementary idea that is used repeatedly.

\begin{lemma}
\label{lemma:matrix:invert}
Consider the matrices of indeterminates $X=(x_{ij})$ where $1\le i\le m$, $1\le j\le n$, and $Y=(y_{ij})$ where $2\le i\le m$, $2\le j\le n$. Set $R=\ZZ[X]$ and $R'=\ZZ[Y]$. Then the map
\[
R'[x_{11},\dots,x_{m1},\ x_{12},\dots,x_{1n}]_{x_{11}}\to R_{x_{11}}\quad\text{ with }\quad y_{ij}\mapsto x_{ij}-\frac{x_{i1}x_{1j}}{x_{11}}
\]
is an isomorphism. Moreover, $R_{x_{11}}$ is a free $R'$-module, and for each $t\ge1$, one has
\[
I_t(X)R_{x_{11}}\ =\ I_{t-1}(Y)R_{x_{11}}
\]
under this isomorphism.
\end{lemma}

\begin{proof}
After inverting the element $x_{11}$, one may perform row operations to transform $X$ into a matrix where $x_{11}$ is the only nonzero entry in the first column. Then, after subtracting appropriate multiples of the first column from other columns, one obtains a matrix
\[
\begin{pmatrix}
x_{11}&0&\dots&0\\
0&x'_{22}&\dots&x'_{2n}\\
\vdots&\vdots&&\vdots\\
0&x'_{m2}&\dots&x'_{mn}
\end{pmatrix}\,
\qquad\text{ where } x'_{ij}\ =\ x_{ij}-\frac{x_{i1}x_{1j}}{x_{11}}\,;
\]
the asserted isomorphism is then $y_{ij}\mapsto x'_{ij}$. 
The ideal $I_t(X)R_{x_{11}}$ is generated by the size $t$ minors of the displayed matrix, and hence equals $I_{t-1}(Y)R_{x_{11}}$. The assertion that $R_{x_{11}}$ is a free $R'$-module follows from the fact that the ring extension
\[
\ZZ\left[x_{ij}-x_{i1}x_{1j}/x_{11}\mid 2\le i\le m,\ 2\le j\le n\right]\ \subset\ \ZZ\left[X,\ 1/x_{11}\right]
\]
is obtained by adjoining indeterminates $x_{11},\dots,x_{m1},x_{12},\dots,x_{1n}$, and inverting $x_{11}$.
\end{proof}

\begin{proof}[Proof of Theorem~\ref{theorem:intro:rational}]
Multiplication by a prime integer $p$ on~$R$ induces the exact sequence
\[
\CD
@>>>H^k_{I_t}(R/pR)@>{\delta}>>H^{k+1}_{I_t}(R)@>{\cdot p}>>H^{k+1}_{I_t}(R)@>>>H^{k+1}_{I_t}(R/pR)@>>>\,,
\endCD
\]
and (1) is precisely the statement that each connecting homomorphism $\delta$ as above is zero. The ideal~$I_tR/pR$ is perfect by Hochster-Eagon \cite{Hochster-Eagon}, i.e., $R/(I_t+pR)$ is a Cohen-Macaulay ring; alternatively, see~\cite[Section~12]{DEP}. By \cite[Proposition~III.4.1]{PS}, it follows that 
\[
H^k_{I_t}(R/pR)\ =\ 0\qquad\text{ if and only if }\ k\neq \height I_t\,.
\]
Thus, to prove (1) and (2), it suffices to prove the injectivity of the map
\begin{equation}
\label{eqn:seq:det}
\CD
H^{\height I_t+1}_{I_t}(R)@>{\cdot p}>>H^{\height I_t+1}_{I_t}(R)\,.
\endCD
\end{equation}
We proceed by induction on $t$. The ideal $I_1$ is generated by a regular sequence, so the injectivity holds when $t=1$ as the modules in~\eqref{eqn:seq:det} are zero.

We claim that the $a$-invariant of the ring $R/(I_t+pR)$ is negative. This follows from the fact that $R/(I_t+pR)$ is $F$-rational, see~\cite[Theorem~7.14]{HH:JAG}; alternatively, the $a$-invariant is computed explicitly in \cite[Corollary~1.5]{BH:ainv} as well as \cite{Grabe}. In particular, one has
\[
\left[H^{\dim R/(I_t+pR)}_{(x_{11},\dots,x_{mn})}(R/(I_t+pR))\right]_0\ =\ 0\,.
\]
By Corollary~\ref{corollary:p:torsion}, it now suffices to show that the map~\eqref{eqn:seq:det} is injective upon inverting each $x_{ij}$, without loss of generality, $x_{11}$. We use the matrix $Y$ as in Lemma~\ref{lemma:matrix:invert} with identifications $y_{ij}=x_{ij}-x_{i1}x_{1j}/x_{11}$, and $R'=\ZZ[Y]$. The ring $R_{x_{11}}$ is a free $R'$-module by Lemma~\ref{lemma:matrix:invert}, so one has an $R'$-module isomorphism $R_{x_{11}}\cong\bigoplus R'$, and so
\[
H^{\height I_t(X)+1}_{I_t(X)}(R_{x_{11}})\ =\ H^{\height I_t(X)+1}_{I_{t-1}(Y)}(R_{x_{11}}) \ \cong\ \bigoplus H^{\height I_t(X)+1}_{I_{t-1}(Y)}(R')\,.
\]
But
\[
\height I_t(X)\ =\ (m-t+1)(n-t+1)\ =\ \height I_{t-1}(Y)\,,
\]
and multiplication by $p$ is injective on
\[
H^{\height I_{t-1}(Y)+1}_{I_{t-1}(Y)}(R')
\]
by the inductive hypothesis. This completes the proof of (1) and (2).

We next verify that $H^{mn-t^2+1}_{I_t}(\ZZ[X])$ is a $\QQ$-vector space under the hypotheses of (3). By (2), it is enough to check that $mn-t^2+1$ is greater than
\[
\height I_t\ =\ (m-t+1)(n-t+1)\,.
\]
After rearranging terms, the desired inequality reads
\[
(t-1)(m+n-2t)\ >\ 0\,,
\]
and the hypotheses on $t$ ensure that this is indeed the case. Hence
\[
H^{mn-t^2+1}_{I_t}(\ZZ[X])\ \cong\ H^{mn-t^2+1}_{I_t}(\QQ[X])\,.
\]
We claim that $H^{mn-t^2+1}_{I_t}(\QQ[X])$ is $\frakm$-torsion; it suffices to check that it vanishes upon inverting, say, $x_{11}$. Using Lemma~\ref{lemma:matrix:invert} as before, one has
\[
H^{mn-t^2+1}_{I_t(X)}(\QQ[X])_{x_{11}}\ \cong\ \bigoplus H^{mn-t^2+1}_{I_{t-1}(Y)}(\QQ[Y])\,,
\]
but these modules are zero since $mn-t^2+1$ is greater than
\[
\ara I_{t-1}(Y)=(m-1)(n-1)-(t-1)^2+1\,.
\]
Hence the support of $H^{mn-t^2+1}_{I_t}(\QQ[X])$ is contained in $\{\frakm\}$; of course, $H^k_{I_t}(\QQ[X])=0$ for all $k>mn-t^2+1$. By the $\calD$-module arguments as in the proof of Theorem~\ref{theorem:comparison}, we have
\[
H^{mn-t^2+1}_{I_t}(\QQ[X])\ \cong\ H^{mn}_\frakm(\QQ[X])^\mu\,,
\]
and it remains to determine the integer $\mu$. It suffices to compute this after base change to~$\CC$, so we work instead with $\CC[X]$. By~\cite[Lemma~2]{Bruns-Schwanzl}, one has
\[
\Hsing^{2mn-t^2}(\CC^{mn}\setminus\Var(I_t)\,;\,\CC)\ \cong\ \CC\,,
\]
and Theorem~\ref{theorem:comparison} now implies that $\mu=1$.
\end{proof}

We examine Theorem~\ref{theorem:intro:rational} for a $2\times3$ matrix of indeterminates:

\begin{example}
\label{example:2x3}
Let $R=\ZZ[u,v,w,x,y,z]$ be a polynomial ring over $\ZZ$. Take $I_2$ to be the ideal generated by the size $2$ minors of the matrix
\[
\begin{pmatrix}
u & v & w\cr
x & y & z
\end{pmatrix}\,.
\]
Let $p$ be a prime integer, and set $\bar{R}=R/pR$. Multiplication by $p$ on $R$ induces the cohomology exact sequence
\[
\CD
@>>>H^2_{I_2}(R)@>\pi>>H^2_{I_2}(\bar{R})@>{\delta}>>H^3_{I_2}(R)@>{\cdot p}>>H^3_{I_2}(R)@>>>0\,,
\endCD
\]
bearing in mind that $H^3_{I_2}(\bar{R})=0$ since $I_2$ is perfect. Theorem~\ref{theorem:intro:rational} implies that the connecting homomorphism $\delta$ is zero i.e., that $\pi$ is surjective; we examine this in elementary terms. Towards this, view $H^2_{I_2}(\bar{R})$ as the direct limit
\[
\dlim_{e\in\NN}\Ext^2_{\bar{R}}\big(\bar{R}/(\Delta_1^{[p^e]},\Delta_2^{[p^e]},\Delta_3^{[p^e]}),\ \bar{R}\big)\,,
\]
where $\Delta_1=vz-wy$, $\Delta_2=wx-uz$, and $\Delta_3=uy-vx$. The complex
\[
\CD
0@>>>\bar{R}^2@>{\left[\begin{smallmatrix}u&x\\v&y\\w&z\end{smallmatrix}\right]}>>\bar{R}^3
@>{\left[\begin{smallmatrix}\Delta_1&\Delta_2&\Delta_3\end{smallmatrix}\right]}>>
\bar{R}@>>>0
\endCD
\]
is a free resolution of $\bar{R}/(\Delta_1,\Delta_2,\Delta_3)$. By the flatness of the Frobenius map, it follows that
\[
\CD
0@>>>\bar{R}^2@>{\left[\begin{smallmatrix}u^{p^e}&x^{p^e}\\v^{p^e}&y^{p^e}\\w^{p^e}&z^{p^e}\end{smallmatrix}\right]}>>\bar{R}^3
@>{\left[\begin{smallmatrix}\Delta_1^{p^e}&\Delta_2^{p^e}&\Delta_3^{p^e}\end{smallmatrix}\right]}>>\bar{R}@>>>0
\endCD
\]
is a free resolution of $\bar{R}/(\Delta_1^{[p^e]},\Delta_2^{[p^e]},\Delta_3^{[p^e]})$ for each $e\ge 1$. Hence $H^2_{I_2}(\bar{R})$ is generated by elements $\alpha_e$ and $\beta_e$ corresponding to the relations
\begin{align*}
u^{p^e}\Delta_1^{p^e}+v^{p^e}\Delta_2^{p^e}+w^{p^e}\Delta_3^{p^e}&\equiv0\mod pR\,,\\
x^{p^e}\Delta_1^{p^e}+y^{p^e}\Delta_2^{p^e}+z^{p^e}\Delta_3^{p^e}&\equiv0\mod pR\,,
\end{align*}
respectively, where $e\ge 1$. As $\pi$ is surjective, these relations must lift to $R$; indeed, in \cite{Singh:MRL}, we constructed the following equational identity:
\begin{multline*}
\sum_{i,j}\binom{k}{i+j}\binom{k+i}{k}\binom{k+j}{k}\left[
u^{k+1}\Delta_1^{2k+1}(-wx)^i(vx)^j\Delta_2^{k-i}\Delta_3^{k-j}\right.\\
\qquad\qquad\qquad\qquad +v^{k+1}\Delta_2^{2k+1}(-uy)^i(wy)^j\Delta_3^{k-i}\Delta_1^{k-j}\\
\left.+w^{k+1}\Delta_3^{2k+1}(-vz)^i(uz)^j\Delta_1^{k-i}\Delta_2^{k-j}\right]=0
\end{multline*}
for each $k\ge0$. Viewed as a relation on the elements $\Delta_1^{2k+1}$, $\Delta_2^{2k+1}$, and $\Delta_3^{2k+1}$, the identity yields an element of $H^2_{I_2}(R)$ for each $k$. Take $k=p^e-1$. Since
\[
\binom{k}{i+j}\binom{k+i}{k}\binom{k+j}{k}\equiv0\mod p\quad\text{ unless }(i,j)=(0,0)\,,
\]
the element of $H^2_{I_2}(R)$ maps to an element of $H^2_{I_2}(\bar{R})$ corresponding to the relation
\[
\big(\Delta_1\Delta_2\Delta_3\big)^{p^e-1}\left[u^{p^e}\Delta_1^{p^e}+v^{p^e}\Delta_2^{p^e}+w^{p^e}\Delta_3^{p^e}\right]\equiv0\mod pR\,,
\]
i.e., precisely to $\alpha_e$. The case of $\beta_e$ is, of course, similar.
\end{example}

\section{The vanishing theorem}
\label{sec:vanishing}

Let $M$ be an $m\times n$ matrix with entries from a commutative Noetherian ring $A$. Set $\fraka$ to be the ideal generated by the size $t$ minors of $M$. By Bruns \cite[Corollary~2.2]{Bruns:MSRI}, the ideal~$\fraka$ can be generated up to radical by $mn-t^2+1$ elements. It follows that $\cd_R(\fraka)$, i.e., the cohomological dimension of $\fraka$, satisfies
\[
\cd_A(\fraka)\ \le\ mn-t^2+1\,.
\]
While this inequality is sharp in general, \cite[Corollary,~page~440]{Bruns-Schwanzl}, we can do better when additional conditions are imposed upon the ring $A$:

\begin{theorem}
\label{theorem:det:vanish}
Let $M=(m_{ij})$ be an $m\times n$ matrix with entries from a commutative Noetherian ring $A$. Let $t$ be an integer with $2\le t\le\min\{m,n\}$ that differs from at least one of~$m$ and~$n$. Set $\fraka$ to be the ideal generated by the size $t$ minors of $M$. Then:
\begin{enumerate}[\quad\rm(1)]
\item The local cohomology module $H^{mn-t^2+1}_\fraka(A)$ is a $\QQ$-vector space, and thus vanishes if the canonical homomorphism $\ZZ\to A$ is not injective.

\item Suppose that $\dim\,A<mn$, or, more generally, that $\dim\,A\otimes_\ZZ\QQ<mn$. Then one has $\cd_A(\fraka)<mn-t^2+1$; in particular, $H^{mn-t^2+1}_\fraka(A)=0$.

\item If the images of $m_{ij}$ in $A\otimes_\ZZ\QQ$ are algebraically dependent over a field that is a subring of $A\otimes_\ZZ\QQ$, then $\cd_A(\fraka)<mn-t^2+1$.
\end{enumerate}
\end{theorem}

\begin{remark}
\label{remark:det:sharp}
The hypotheses of the theorem exclude $t=1$ and $t=m=n$ since, in these cases, assertion (1) need not hold.

The cohomological dimension bounds in (2) and (3) are sharp: take $R=\QQ[X]$ to be the ring of polynomials in an $m\times n$ matrix of indeterminates $X$, and set~$A=R/x_{11}R$; note that $\dim\,A<mn$. Let $t$ be as in Theorem~\ref{theorem:det:vanish}. Multiplication by~$x_{11}$ on $R$ induces the local cohomology exact sequence
\[
\CD
@>>>H^{mn-t^2}_{I_tA}(A)@>\delta>>H^{mn-t^2+1}_{I_t}(R)@>x_{11}>>H^{mn-t^2+1}_{I_t}(R)@>>>0\,.
\endCD
\]
By Theorem~\ref{theorem:intro:rational}, the support of $H^{mn-t^2+1}_{I_t}(R)$ is precisely the homogenous maximal ideal of $R$, so $\image\delta=\ker x_{11}$ is nonzero. It follows that $H^{mn-t^2}_{I_tA}(A)$ is nonzero. 
\end{remark}

\begin{proof}[Proof of Theorem~\ref{theorem:det:vanish}]
Set $R$ to be the polynomial ring $\ZZ[X]$, where $X$ is an $m\times n$ matrix of indeterminates. By Theorem~\ref{theorem:intro:rational} we have
\[
H^{mn-t^2+1}_{I_t}(R)\ \cong\ H^{mn}_\frakm(\QQ\otimes_\ZZ R),
\]
where $\frakm=(x_{11},\dots,x_{mn})R$. Let $R\to A$ be the ring homomorphism with $x_{ij}\mapsto m_{ij}$; the extended ideal $I_tA$ equals $\fraka$. Since $I_t$ is generated up to radical by $mn-t^2+1$ elements, and $\frakm$ by $mn$ elements, Lemma~\ref{lemma:basechange} provides the first and the third of the isomorphisms below:
\begin{multline}
\label{eqn:iso}
H^{mn-t^2+1}_\fraka(A)\ \cong\ H^{mn-t^2+1}_{I_t}(R)\otimes_RA\
\cong\ H^{mn}_\frakm(\QQ\otimes_\ZZ R)\otimes_RA\
\cong\ H^{mn}_\frakm(\QQ\otimes_\ZZ A)\,.
\end{multline}
It follows that $H^{mn-t^2+1}_\fraka(A)$ is a $\QQ$-vector space, which settles (1).

For (2), if $\dim\,A\otimes_\ZZ\QQ<mn$, then $H^{mn}_\frakm(A\otimes_\ZZ\QQ)$ vanishes since the cohomological dimension is bounded above by the Krull dimension of the ring. Thus, by~\eqref{eqn:iso},
\[
H^{mn-t^2+1}_\fraka(A)\ =\ 0\,.
\]
Since $\fraka$ can be generated up to radical by $mn-t^2+1$ elements, $H^k_\fraka(A)$ also vanishes for all integers $k$ with $k>mn-t^2+1$. Hence, $\cd_A(\fraka)<mn-t^2+1$.

For (3), let $\FF$ be the field, and set $B$ to be the $\FF$-subalgebra of~$A\otimes_\ZZ\QQ$ generated by the images of $m_{ij}$. Take $\frakb$ to be the ideal of $B$ generated by the size $t$ minors. Then $\dim\,B<mn$, so (2) gives $H^{mn-t^2+1}_\frakb(B)=0$. Using (1) along with Lemma~\ref{lemma:basechange}, it follows that
\[
H^{mn-t^2+1}_\fraka(A)\ \cong \ H^{mn-t^2+1}_\fraka(A\otimes_\ZZ\QQ)\ \cong \ H^{mn-t^2+1}_\frakb(B)\otimes_B(A\otimes_\ZZ\QQ)\ =\ 0\,.
\qedhere
\]
\end{proof}

\section{Pfaffians of alternating matrices}
\label{sec:alt}

We prove the analogues of Theorem~\ref{theorem:intro:rational} and Theorem~\ref{theorem:det:vanish} for Pfaffians of alternating matrices. Let $t$ be an even integer. The ideal generated by the Pfaffians of the $t\times t$ diagonal submatrices of an $n\times n$ alternating matrix of indeterminates has height
\[
\binom{n-t+2}{2}\,,
\]
see for example~\cite[Section~2]{Jozefiak-Pragacz}, and its arithmetic rank is
\[
\binom{n}{2}-\binom{t}{2}+1
\]
by Barile,~\cite[Theorem~4.1]{Barile:ara}. We need the following result, which is the analogue of Lemma~\ref{lemma:matrix:invert} for alternating matrices; see \cite[Lemma~1.2]{Jozefiak-Pragacz} or \cite[Lemma~1.3]{Barile:ara}:

\begin{lemma}
\label{lemma:pfaffian:invert}
Let $X$ be an $n\times n$ alternating matrix of indeterminates; set $R=\ZZ[X]$. Then there exists an $(n-2)\times(n-2)$ generic alternating matrix $Y$ with entries from $R_{x_{12}}$, such that $R_{x_{12}}$ is a free $\ZZ[Y]$-module, and
\[
P_t(X)R_{x_{12}}\ =\ P_{t-2}(Y)R_{x_{12}}\quad\text{ for each even integer }t\ge4\,.
\]
\end{lemma}

\begin{theorem}
\label{theorem:pfaffian:rational}
Let $R=\ZZ[X]$ be a polynomial ring, where $X$ is an $n\times n$ alternating matrix of indeterminates. Let $t$ be an even integer, and let $P_t$ denote the ideal generated by the Pfaffians of the size $t$ diagonal submatrices of~$X$. Then:
\begin{enumerate}[\quad\rm(1)]
\item $H^k_{P_t}(R)$ is a torsion-free $\ZZ$-module for all integers $k$.
\item If $k$ differs from the height of $P_t$, then $H^k_{P_t}(R)$ is a $\QQ$-vector space.
\item Let $\frakm$ be the homogeneous maximal ideal of $\QQ[X]$. If $2<t<n$, then
\[
H^{\binom{n}{2}-\binom{t}{2}+1}_{P_t}(\ZZ[X])\ \cong\ H^{\binom{n}{2}}_\frakm(\QQ[X])\,.
\]
\end{enumerate}
\end{theorem}

\begin{proof}
We follow the logical structure of the proof of Theorem~\ref{theorem:intro:rational}. Let $p$ be a prime integer. The ring $R/(P_t+pR)$ is Cohen-Macaulay by~\cite{Kleppe-Laksov} or~\cite{Marinov1, Marinov2}, so the module~$H^k_{P_t}(R/pR)$ vanishes for $k\neq\height P_t$ by \cite[Proposition~III.4.1]{PS}. For (1) and (2), it thus suffices to prove the injectivity of the map
\begin{equation}
\label{eqn:seq:alt}
\CD
H^{\height P_t+1}_{P_t}(R)@>{\cdot p}>>H^{\height P_t+1}_{P_t}(R)\,,
\endCD
\end{equation}
and this is by induction on $t$, the case $t=2$ being trivial since $P_2$ is generated by a regular sequence. The $a$-invariant of $R/(P_t+pR)$ is computed explicitly in \cite[Corollary~1.7]{BH:ainv}; alternatively, $(R/P_t)\otimes_\ZZ\QQ$ is the ring of invariants of the symplectic group---which is linearly reductive in the case of characteristic zero---and hence has rational singularities; using Lemma~\ref{lemma:ainv}, it follows that the $a$-invariant of the ring $R/(P_t+pR)$ is negative.

Using Lemma~\ref{lemma:pfaffian:invert}, the inductive hypothesis implies that~\eqref{eqn:seq:alt} is injective upon inverting $x_{12}$, equivalently, any $x_{ij}$. But then Corollary~\ref{corollary:p:torsion} yields the injectivity of~\eqref{eqn:seq:alt}.

For (3), note that $H^{\binom{n}{2}-\binom{t}{2}+1}_{P_t}(\ZZ[X])$ is a $\QQ$-vector space by (2), since the hypothesis
\[
(t-2)(n-t)\ >\ 0\
\]
is equivalent to
\[
\binom{n}{2}-\binom{t}{2}+1\ >\ \binom{n-t+2}{2}\,.
\]
To verify that $H^{\binom{n}{2}-\binom{t}{2}+1}_{P_t}(\QQ[X])$ is $\frakm$-torsion, it suffices to check that
\[
H^{\binom{n}{2}-\binom{t}{2}+1}_{P_t}(\QQ[X])_{x_{12}}\ =\ 0\,,
\]
and this follows from Lemma~\ref{lemma:pfaffian:invert} since
\[
\binom{n}{2}-\binom{t}{2}+1\ >\ \ara P_{t-2}(Y)\ =\ \binom{n-2}{2}-\binom{t-2}{2}+1\,.
\]
By Theorem~\ref{theorem:comparison}, it follows that
\[
H^{\binom{n}{2}-\binom{t}{2}+1}_{P_t}(\QQ[X])\ \cong\ H^{\binom{n}{2}}_\frakm(\QQ[X])^\mu\,,
\]
where $\mu$ is the rank of the singular cohomology group
\[
\Hsing^{2\binom{n}{2}-\binom{t}{2}}(L\setminus\Var(P_t)\,;\,\CC)
\]
as a complex vector space, where $L$ is the affine space $\CC^{\binom{n}{2}}$. The computation of this cohomology follows entirely from \cite{Barile:ara}; however, since it is not explicitly recorded, we include a sketch for the convenience of the reader. The cohomology groups below are with coefficient group $\CC$.

Let $V=\Var(P_t)$ and $\tilde{V}=\Var(P_{t+2})$. Then $\tilde{V}\setminus V$ is smooth by \cite[Theorem~17]{Kleppe-Laksov}, and of complex dimension 
\[
\binom{n}{2}-\binom{n-t}{2}\,.
\]
Consider the exact sequence of cohomology with compact support:
\[
\CD
@>>>H^{\binom{t}{2}}_c(L\setminus\tilde{V})@>>>H^{\binom{t}{2}}_c(L\setminus V)@>>>H^{\binom{t}{2}}_c(\tilde{V}\setminus V)@>>>H^{\binom{t}{2}+1}_c(L\setminus\tilde{V})@>>>\,.
\endCD
\]
We claim that the middle map is an isomorphism; for this, it suffices to prove that
\[
H^{\binom{t}{2}}_c(L\setminus\tilde{V})\ =\ 0\ =\ H^{\binom{t}{2}+1}_c(L\setminus\tilde{V})\,.
\]
By Poincar\'e duality, this is equivalent to
\[
\Hsing^{2\binom{n}{2}-\binom{t}{2}}(L\setminus\tilde{V})\ =\ 0\ =\ \Hsing^{2\binom{n}{2}-\binom{t}{2}-1}(L\setminus\tilde{V})\,, 
\]
which follows from Lemma~\ref{lemma:ara} since $P_{t+2}$ has arithmetic rank $\binom{n}{2}-\binom{t+2}{2}+1$.

Using Poincar\'e duality once again, we have
\[
\Hsing^{2\binom{n}{2}-\binom{t}{2}}(L\setminus V)\ \cong\ \Hsing^{2\binom{n}{2}-2\binom{n-t}{2}-\binom{t}{2}}(\tilde{V}\setminus V)\,.
\]
By \cite[page~73]{Barile:ara}, the space $\tilde{V}\setminus V$ is a fiber bundle over the Grassmannian $G_{n-t,n}$, with the fiber being the space $\Alt(t)$ of invertible alternating matrices of size $t$; the latter space is homotopy equivalent to a compact, connected, orientable manifold of real dimension $\binom{t}{2}$. Since $G_{n-t,n}$ is simply connected, the Leray spectral sequence
\[
E^{p,q}_2\ =\ \Hsing^p(G_{n-t,n}\,;\,\Hsing^q(\Alt(t)))\ \implies \Hsing^{p+q}(\tilde{V}\setminus V)
\]
shows that
\[
\Hsing^{2\binom{n}{2}-2\binom{n-t}{2}-\binom{t}{2}}(\tilde{V}\setminus V)\ \cong\ \CC\,,
\]
and it follows that $\mu=1$.
\end{proof}

We next record the vanishing theorem for local cohomology supported at Pfaffian ideals:

\begin{theorem}
Let $M=(m_{ij})$ be an $n\times n$ alternating matrix with entries from a commutative Noetherian ring $A$. Let $t$ be even with $2<t<n$, and set $\fraka$ to be the ideal generated by the Pfaffians of the size $t$ diagonal submatrices of $M$. Set $c=\binom{n}{2}-\binom{t}{2}+1$. Then:
\begin{enumerate}[\quad\rm(1)]
\item The local cohomology module $H^c_\fraka(A)$ is a $\QQ$-vector space, and thus vanishes if the canonical homomorphism $\ZZ\to A$ is not injective.

\item If $\dim\,A\otimes_\ZZ\QQ<\binom{n}{2}$, or, more generally, if the images of $m_{ij}$ in $A\otimes_\ZZ\QQ$ are algebraically dependent over a field that is a subring of $A\otimes_\ZZ\QQ$, then $\cd_A(\fraka)<c$.
\end{enumerate}
\end{theorem}

\begin{proof}
Set $R=\ZZ[X]$, where $X$ is an $n\times n$ alternating matrix of indeterminates. Define an $R$-algebra structure on~$A$ using $x_{ij}\mapsto m_{ij}$. The theorem now follows from
\[
H^c_{P_t}(R)\ \cong\ H^{\binom{n}{2}}_\frakm(\QQ\otimes_\ZZ R),
\]
using arguments as in the proof of Theorem~\ref{theorem:det:vanish}.
\end{proof}

\begin{remark}
Once again, the bound on $\cd_A(\fraka)$ is sharp: Take $R=\QQ[X]$ to be a polynomial ring in an $n\times n$ alternating matrix of indeterminates $X$. As in Remark~\ref{remark:det:sharp}, set $A=R/x_{12}R$. Then $H^{c-1}_{P_tA}(A)$ is nonzero. 
\end{remark}

\section{Minors of symmetric matrices}
\label{sec:sym}

We prove the analogue of Theorem~\ref{theorem:intro:rational} for minors of symmetric matrices, and also the analogue of Theorem~\ref{theorem:det:vanish} in the case of minors of odd size; the corresponding result is not true for even sized minors, see Remark~\ref{remark:sym:even}. The ideal $I_t$ generated by the size $t$ minors of an $n\times n$ symmetric matrix of indeterminates has height
\[
\binom{n-t+2}{2}\,,
\]
see, for example~\cite[Section~2]{Jozefiak}. By~\cite[Theorems~3.1,~5.1]{Barile:ara}, the arithmetic rank of $I_t$ is 
\[
\ara I_t\ =\ \begin{cases}
\binom{n}{2}-\binom{t}{2}+1 & \text{if the characteristic equals $2$, and $t$ is even,}\\
\binom{n+1}{2}-\binom{t+1}{2}+1 & \text{else}.
\end{cases}
\]
For a symmetric matrix of indeterminates over a field of characteristic zero, the cohomological dimension of the ideal $I_t$ is
\[
\cd(I_t)\ =\ \begin{cases}
\binom{n+1}{2}-\binom{t+1}{2}+1 & \text{if $t$ is odd,}\\
\binom{n}{2}-\binom{t}{2}+1 & \text{if $t$ is even},
\end{cases}
\]
as proved in \cite[Theorem~6.3]{Barile:ara} and \cite[(1.6)]{RaicuWeyman:symm} respectively.

\begin{theorem}
\label{theorem:symmetric:rational}
Let $R=\ZZ[X]$ be a polynomial ring, where $X$ is a symmetric matrix of indeterminates. Let $I_t$ denote the ideal generated by the size $t$ minors of $X$. Then:
\begin{enumerate}[\quad\rm(1)]
\item $H^k_{I_t}(R)$ is a torsion-free $\ZZ$-module for all integers $t,k$.
\item If $k$ differs from the height of $I_t$, then $H^k_{I_t}(R)$ is a $\QQ$-vector space.
\item Let $\frakm$ be the homogeneous maximal ideal of $\QQ[X]$. If $t$ is odd with $1<t<n$, then
\[
H^{\binom{n+1}{2}-\binom{t+1}{2}+1}_{I_t}(\ZZ[X])\ \cong\ H^{\binom{n+1}{2}}_\frakm(\QQ[X])\,.
\]
\end{enumerate}
\end{theorem}

The analogue of Lemma~\ref{lemma:matrix:invert} for symmetric matrices is the following; for a proof, see \cite[Lemme~2]{Micali-Villamayor} or \cite[Lemma~1.1]{Jozefiak} or \cite[Lemma~1.2]{Barile:ara}.

\begin{lemma}
\label{lemma:sym:invert}
Let $X$ be an $n\times n$ symmetric matrix of indeterminates. Set $R=\ZZ[X]$ and $\Delta=x_{11}x_{22}-x_{12}^2$. Then:
\begin{enumerate}[\quad\rm(1)]
\item There exists an $(n-1)\times(n-1)$ generic symmetric matrix $Y$ with entries from $R_{x_{11}}$ such that $R_{x_{11}}$ is a free $\ZZ[Y]$-module, and
\[
I_t(X)R_{x_{11}}\ =\ I_{t-1}(Y)R_{x_{11}}\quad\text{ for each }t\ge2\,.
\]
\item There exists an $(n-2)\times(n-2)$ generic symmetric matrix $Y'$ with entries from $R_\Delta$ such that $R_\Delta$ is a free $\ZZ[Y']$-module, and
\[
I_t(X)R_\Delta\ =\ I_{t-2}(Y')R_\Delta\quad\text{ for each }t\ge3\,.
\]
\end{enumerate}
\end{lemma}

\begin{proof}[Proof of Theorem~\ref{theorem:symmetric:rational}]
For the most part, the proof is similar to that of Theorem~\ref{theorem:intro:rational} and Theorem~\ref{theorem:pfaffian:rational}: The ring $R/(I_t+pR)$ is Cohen-Macaulay by Kutz~\cite{Kutz}, so $H^k_{I_t}(R/pR)$ vanishes for $k\neq\height I_t$ by \cite[Proposition~III.4.1]{PS}. For (1) and (2), it suffices to prove the injectivity of the map
\begin{equation}
\label{eqn:seq:sym}
\CD
H^{\height I_t+1}_{I_t}(R)@>{\cdot p}>>H^{\height I_t+1}_{I_t}(R)\,,
\endCD
\end{equation}
and this is by induction on $t$, the case $t=1$ being trivial since $I_1$ is generated by a regular sequence. The $a$-invariant of $R/(I_t+pR)$ is computed in \cite{Barile:ainv} as well as \cite[Section~2.2]{Conca}; alternatively, $(R/I_t)\otimes_\ZZ\QQ$ is the ring of invariants of the orthogonal group, and hence has rational singularities, and so $R/(I_t+pR)$ has a negative $a$-invariant using Lemma~\ref{lemma:ainv}.

Using Lemma~\ref{lemma:sym:invert}, the inductive hypothesis implies that~\eqref{eqn:seq:sym} is injective upon inverting $x_{11}$ as well as upon inverting $\Delta$. The radical of the ideal generated by the elements~$x_{ii}$ for $1\le i\le n$ and $x_{jj}x_{kk}-x_{jk}^2$ for $j<k$ is $(x_{11},x_{12},\dots,x_{nn})$, so the map~\eqref{eqn:seq:sym} is indeed injective by Corollary~\ref{corollary:p:torsion}.

For (3), note that $H^{\binom{n+1}{2}-\binom{t+1}{2}+1}_{I_t}(\ZZ[X])$ is a $\QQ$-vector space, since $(t-1)(n-t)>0$ ensures that
\[
\binom{n+1}{2}-\binom{t+1}{2}+1\ >\ \binom{n-t+2}{2}\,.
\]
To verify that $H^{\binom{n+1}{2}-\binom{t+1}{2}+1}_{I_t}(\QQ[X])$ is $\frakm$-torsion, it suffices to check that
\[
H^{\binom{n+1}{2}-\binom{t+1}{2}+1}_{I_t}(\QQ[X])_{x_{11}}\ =\ 0\ =\ H^{\binom{n+1}{2}-\binom{t+1}{2}+1}_{I_t}(\QQ[X])_\Delta\,.
\]
By Lemma~\ref{lemma:sym:invert}, it is enough to check that
\[
\binom{n+1}{2}-\binom{t+1}{2}+1\ >\ \ara I_{t-1}(Y)\ =\ \binom{n}{2}-\binom{t}{2}+1
\] 
and that
\[
\binom{n+1}{2}-\binom{t+1}{2}+1\ >\ \ara I_{t-2}(Y')\ =\ \binom{n-1}{2}-\binom{t-1}{2}+1\,,
\] 
which is indeed the case. Theorem~\ref{theorem:comparison} now implies that
\[
H^{\binom{n+1}{2}-\binom{t+1}{2}+1}_{I_t}(\QQ[X])\ \cong\ H^{\binom{n+1}{2}}_\frakm(\QQ[X])^\mu\,,
\]
with $\mu$ being the rank of the singular cohomology group
\[
\Hsing^{2\binom{n+1}{2}-\binom{t+1}{2}}(L\setminus\Var(P_t)\,;\,\CC)\,,
\]
where $L=\CC^{\binom{n+1}{2}}$. This, again, follows from \cite{Barile:ara}, though we sketch a proof:

Let $V=\Var(I_t)$ and $\tilde{V}=\Var(I_{t+1})$. Then $\tilde{V}\setminus V$ is smooth by \cite[Theorem~2.2]{Barile:ara}, and of complex dimension 
\[
\binom{n+1}{2}-\binom{n-t+1}{2}\,.
\]
Consider the exact sequence of cohomology with compact support:
\[
\CD
H^{\binom{t+1}{2}}_c(L\setminus\tilde{V})@>>>H^{\binom{t+1}{2}}_c(L\setminus V)@>>>H^{\binom{t+1}{2}}_c(\tilde{V}\setminus V)@>>>H^{\binom{t+1}{2}+1}_c(L\setminus\tilde{V})\,.
\endCD
\]
The ideal $I_{t+1}$ has arithmetic rank $\binom{n+1}{2}-\binom{t+2}{2}+1$, so Lemma~\ref{lemma:ara} implies that
\[
\Hsing^{2\binom{n+1}{2}-\binom{t+1}{2}}(L\setminus\tilde{V})\ =\ 0\ =\
\Hsing^{2\binom{n+1}{2}-\binom{t+1}{2}-1}(L\setminus\tilde{V})\,.
\]
By Poincar\'e duality, one then has
\[
H^{\binom{t+1}{2}}_c(L\setminus\tilde{V})\ =\ 0\ =\ H^{\binom{t+1}{2}+1}_c(L\setminus\tilde{V})\,.
\]
Thus, Poincar\'e duality gives
\[
\Hsing^{2\binom{n+1}{2}-\binom{t+1}{2}}(L\setminus V)\ \cong\ \Hsing^{2\binom{n+1}{2}-2\binom{n-t+1}{2}-\binom{t+1}{2}}(\tilde{V}\setminus V)\,.
\]
By \cite[page~68]{Barile:ara}, the space $\tilde{V}\setminus V$ is a fiber bundle over the Grassmannian $G_{n-t,n}$, with the fiber being the space $\Sym(t)$ of invertible symmetric matrices of size $t$; the latter space is homotopy equivalent to a compact, connected manifold of real dimension $\binom{t+1}{2}$, and when $t$ is odd, the manifold is orientable. The Leray spectral sequence
\[
E^{p,q}_2\ =\ \Hsing^p(G_{n-t,n}\,;\,\Hsing^q(\Sym(t)))\ \implies \Hsing^{p+q}(\tilde{V}\setminus V)
\]
now gives
\[
\Hsing^{2\binom{n+1}{2}-2\binom{n-t+1}{2}-\binom{t+1}{2}}(\tilde{V}\setminus V)\ \cong\ \CC\,,
\]
completing the proof.
\end{proof}

\begin{theorem}
Let $M=(m_{ij})$ be an $n\times n$ symmetric matrix with entries from a commutative Noetherian ring $A$. Let $t$ be an odd integer with $1<t<n$, and set $\fraka$ to be the ideal generated by the size $t$ minors of $M$. Set $c=\binom{n+1}{2}-\binom{t+1}{2}+1$. Then:
\begin{enumerate}[\quad\rm(1)]
\item The local cohomology module $H^c_\fraka(A)$ is a $\QQ$-vector space, and thus vanishes if the canonical homomorphism $\ZZ\to A$ is not injective.

\item If $\dim\,A\otimes_\ZZ\QQ<\binom{n+1}{2}$, or, more generally, if the images of $m_{ij}$ in $A\otimes_\ZZ\QQ$ are algebraically dependent over a field that is a subring of $A\otimes_\ZZ\QQ$, then $\cd_A(\fraka)<c$.
\end{enumerate}
\end{theorem}

\begin{proof}
The proof is similar to that of Theorem~\ref{theorem:det:vanish}.
\end{proof}

\begin{remark}
The bound on $\cd_A(\fraka)$ above is sharp: Take $R=\QQ[X]$ to be a polynomial ring in an $n\times n$ symmetric matrix of indeterminates. The module $H^c_{I_t}(R)$ is $\frakm$-torsion for $t$ odd, and it follows as in Remark~\ref{remark:det:sharp}, that $H^{c-1}_{I_tA}(A)$ is nonzero for $A=R/x_{11}R$. 
\end{remark}

\begin{remark}
\label{remark:sym:even}
Let $R=\QQ[X]$ be a polynomial ring in an $n\times n$ symmetric matrix of indeterminates, and consider the ideal $I_2$ generated by the size $2$ minors of $X$. Then $\cd_R(I_2)=\binom{n}{2}$ by \cite[Example~4.6]{Ogus}. Set $A=R/(x_{11},x_{22},\dots,x_{nn})$. Then the ideal $I_2A$ is primary to the homogeneous maximal ideal of $A$, and hence
\[
\cd_A(\fraka)\ =\ \dim\,A\ =\ \binom{n}{2}\,. 
\]
Thus, while $\dim\,A<\dim\,R$, we have $\cd_A(I_2A)=\cd_R(I_2)$.
\end{remark}

\section{A question on arithmetic rank}
\label{sec:ara}

The vanishing result, Theorem~\ref{theorem:det:vanish}, raises the following question:

\begin{question}
Let $A$ be a polynomial ring over a field, and $\fraka$ the ideal generated by the size~$t$ minors of an $m\times n$ matrix with entries from $A$. Suppose $\dim\,A<mn$, and that~$t$ differs from at least one of $m,n$. Can $\fraka$ be generated up to radical by $mn-t^2$ elements?
\end{question}

There are, of course, corresponding questions when $M$ is a symmetric or alternating matrix. While we admittedly have no approach to these questions, we record two examples:

\begin{example}
This is an example from~\cite{Barile:Colloq}. Let $A$ be the polynomial ring $\FF[v,w,x,y,z]$, and let $\fraka$ be the ideal generated by the size two minors of
\[
\begin{pmatrix}
0 & v & w\cr
x & y & z
\end{pmatrix}\,,
\]
i.e., $\fraka=(vx,\ wx,\ vz-wy)$. Then $\ara\fraka=2$, since $\height\fraka=2$, and $\fraka$ is the radical of the ideal generated by
\[
f=wx^2+z(vz-wy)\quad\text{ and }\quad g=vx^2+y(vz-wy)\,;
\]
to see this, note that $vf-wg=(vz-wy)^2$. 
\end{example}

The following example, and generalizations, may be found in \cite{Valla}; see also \cite{BV, Barile:scrolls}.

\begin{example}
Let $A$ be the polynomial ring $\FF[u,v,w,x,y]$, and let $\fraka$ be the ideal generated by the size two minors of
\[
\begin{pmatrix}
u & v & w\cr
v & x & y
\end{pmatrix}\,.
\]
Then, again, $\ara\fraka=2$, since $\fraka$ is the radical of the ideal generated by $v^2-ux$ and
\[
\det\begin{pmatrix}
u & v & w\cr
v & x & y\cr
w & y & 0
\end{pmatrix}\,,
\]
see, for example, \cite[Example~2.2]{Valla}.
\end{example}

For some recent results concerning matrices of linear forms, see~\cite{BCMM}.

\section*{Acknowledgments}

We thank Bhargav Bhatt, Manuel Blickle, Winfried Bruns, Linquan Ma, Claudia Miller, Claudio Procesi, Matteo Varbaro, and Wenliang Zhang for valuable discussions.


\end{document}